%% file: main.tex
\documentclass[11pt,a4paper]{article}

\usepackage{graphicx}
\usepackage{subcaption} 
\usepackage{amsmath, amsfonts, amssymb}
\usepackage{enumerate}
\usepackage{longtable}
\usepackage{multirow}
\usepackage{color}
\usepackage{url}
\usepackage{caption}
\usepackage{subcaption}
\usepackage{enumitem}
\usepackage{dsfont}
\usepackage{algorithm}        
\usepackage{algpseudocode}
\usepackage{mathtools}
\usepackage{hyperref}
\usepackage[normalem]{ulem}

\newtheorem{theorem}{Theorem}[section]

\newtheorem{corollary}{Corollary}[section]

\newtheorem{definition}{Definition}[section]

\newtheorem{lemma}{Lemma}[section]

\newtheorem{proposition}{Proposition}[section]
\newtheorem{remark}{Remark}[section]

\newtheorem{assumption}{Assumption}[section]

\newenvironment{proof}[1][Proof]{%
  \par\noindent\textbf{#1.}\ %
}{%
  \par\vspace{-0.5 ex}            
  \noindent\hfill\rule{0.5em}{0.5em}\par
  \vspace{1ex}
}

\newcommand{\R}{\mathbb{R}}

\newcommand{\N}{\ensuremath \mathbb{N}}

\newcommand{\F}{\mathcal{F}}

\newcommand{\Pb}{\mathbb{P}}

\newcommand{\me}{\mathbb{E}}

\newcommand{\f}{\tilde{f}}
\newcommand{\D}{\Delta}
\newcommand{\Imax}{\bar{\imath}}

\newcommand{\Dexp}{p}
\newcommand{\aexp}{\Dexp/(\Dexp-1)}
\newcommand{\goldelta}{r}

\newcommand{\Xk}{X_k}

\newcommand{\EE}{\mathbb{E}}

\title{Extrapolation-based Direct Search for Nonsmooth Stochastic Zeroth-Order Optimization}

	\author{
    A.~Palmieri \thanks{Dipartimento di Matematica, Universit\`a
		di Padova, Italy
		(\tt{anthony.palmieri@studenti.unipd.it}).}
	\and
	F.~Rinaldi \thanks{Dipartimento di Matematica, Universit\`a
		di Padova, Italy
		(\tt{rinaldi@math.unipd.it}).}
    \and
    S.~Shashaani
    \thanks{Edward P. Fitts Department of Industrial and System Engineering, North Carolina State University, Raleigh, NC, USA
		(\tt{sshasha2@ncsu.edu}).}
}

\begin{document}
	\maketitle

\input{sections/abstract}

\input{sections/intro}


\input{sections/algorithm}

\input{sections/assumptions_and_preliminaries}
\input{sections/convergence}


\input{sections/complexity_goldstein}

\input{sections/numerical_results}

\input{sections/conclusions}
\appendix
\input{sections/appendix}

\input{sections/appendix_estimates}

\bibliographystyle{plain}
\bibliography{sdse}
	
\end{document}

%% file: sections/abstract.tex
We propose and analyze a stochastic direct-search method for unconstrained
zeroth-order minimization of locally Lipschitz, possibly nonsmooth, objectives.
The method combines random polling directions with a stochastic extrapolating
line search based on a sufficient-decrease test of order \(\Dexp\). Under
conditional accuracy assumptions on the stochastic estimates, 
we prove almost-sure convergence to Clarke
stationary points. We further establish an expected iteration complexity
bound. Specifically, using a supermartingale stopping-time argument, we prove
that
$\mathcal O\left(
        \max\left\{
            \goldelta^{-\Dexp},
            \varepsilon^{-\Dexp/(\Dexp-1)}
        \right\}
    \right)
$
iterations are sufficient in expectation to reach an
\((\goldelta,\varepsilon)\)-Goldstein stationary point. Moreover, we derive a corresponding expected tested-point complexity bound of order
\(\mathcal O\bigl(\varepsilon^{1-n}
\max\{r^{-p},\varepsilon^{-p/(p-1)}\}\bigr)\). To the best of
our knowledge, this is the first convergence and expected-complexity analysis
for an \emph{extrapolation-based direct-search} method in a \emph{nonsmooth
stochastic} setting. Numerical experiments on a DFO benchmark suite highlight competitive performance against well-established stochastic direct-search methods.


%% file: sections/intro.tex
\section{Introduction} \label{sec:intro}
We study the unconstrained problem
\begin{equation}\label{gen_prob}
  \min_{x\in\R^n} f(x),
\end{equation}
where $f:\R^n\to\R$ is a \emph{locally} Lipschitz (possibly nonsmooth) function that cannot be evaluated exactly. 
Instead, we only have access  to a \emph{stochastic
zeroth-order oracle} that returns a random estimate $\tilde f(x)$ of $f(x)$ at any query point $x$. For example, 
the stochastic estimate can be modeled as a random variable
parameterized by $x$, namely
\[
  \tilde f(x)\;=\;F(x,\zeta),\qquad \zeta\sim\mathbb P,
\]
i.e., the oracle draws a random seed $\zeta$  from an underlying probability space
$(\Omega,\mathcal F,\mathbb P)$ and outputs $F(x,\zeta)$. 
Throughout the paper, all algorithmic quantities are random variables adapted to
the natural filtration induced by the oracle and the algorithmic randomness (precise assumptions on
the estimator, e.g., probabilistic accuracy, tails, and independence are stated in
Section~\ref{sec:assumptions}).
Our goal is to design and analyze a \emph{derivative-free} method that combines random dense directions with a
stochastic line search, proving convergence to Clarke stationary points and establishing expected
complexity bounds in terms of oracle calls.

\subsection{Related Work} \label{subsec:related_work}
\input{sections/related_work}

\subsection{Contributions} \label{subsec:contributions}
\input{sections/contributions}
\subsection{Paper Structure} \label{subsec:paper_structure}
The paper proceeds as follows. 
Section~2 presents the algorithm, the stochastic line search, and notation.
Section~3 states the standing assumptions (compactness, conditional independence, and $\Dexp$-tail).
Section~4 proves a uniformly positive conditional expected merit-function decrease proportional  to $\Delta_k^p$, shows $\sum_k \Delta_k^{\Dexp}<\infty$ for $\Dexp >1$ almost surely and hence $\Delta_k\to 0$ almost surely. This establishes convergence to Clarke stationary point almost surely. Section 5 proves that the algorithm reaches an $(\goldelta,\varepsilon)$-Goldstein stationary point in expected $\mathcal O\left(\max\left\{\goldelta^{-\Dexp}, \varepsilon^{-\frac{\Dexp}{\Dexp-1}}\right\}\right)$ iterations 
using a supermartingale stopping-time analysis. Furthermore, it derives
the corresponding expected tested-point complexity.
Section~6 presents numerical experiments showing that the proposed method is competitive with state-of-the-art stochastic derivative-free algorithms on a standard nonsmooth benchmark.
Section 7 finally summarizes the theoretical results and discusses possible future extensions of the method. The appendix gathers auxiliary proofs.

%% file: sections/related_work.tex
Derivative-free optimization (DFO) focuses on  problems whose derivative information is unavailable, unreliable, or simply too expensive to obtain, requiring algorithms to work only with (possibly noisy) function values. Two main approaches are common in this context, which are the model-based methods and the direct-search methods. In model-based methods, one builds at each iteration a local surrogate  intended to approximate $f$ on a neighborhood of the current point, proposes a step by (approximately) minimizing the model inside that neighborhood, and accepts or rejects it by comparing predicted and observed decrease. Direct-search methods, by contrast, do not form an explicit model, but rather evaluate $f$ on a structured set of trial points 
and update the current point based on a tailored acceptance test (e.g., sufficient-decrease of the objective function).
The deterministic DFO theory is well established; see, e.g., the monograph of Conn, Scheinberg, and Vicente~\cite{MR2487816}, and the survey of Larson, Menickelly, and Wild~\cite{MR3963507} for further details. 
Both model-based and direct-search schemes extend to stochastic zeroth-order oracles, but the analysis must then keep track of how estimation error affects acceptance decisions and step-size updates.\\
On the model-based side, trust-region methods based on probabilistic accuracy conditions~\cite{MR3245880} demonstrate that much of the classical trust-region mechanism remains valid when local models and function estimates are sufficiently accurate with 
high probability, and martingale-based arguments play a central role in the theoretical analysis. This line of work ultimately led to the development of STORM-type frameworks, which systematically integrate randomized local models with stochastic function evaluations within a trust-region mechanism~\cite{MR3800867}. Expected-complexity guarantees for STORM-type methods can be derived through stopping-time and supermartingale techniques~\cite{MR4151319}, and this analysis has been adapted both to random-subspace trust-region schemes for large-scale problems~\cite{MR4777851} and to stochastic line-search mechanisms~\cite{MR4060460}. In a separate framework for derivative-free stochastic trust-region schemes, adaptive sampling allows to 
drive down the estimation error according to a measure of stationarity, ensuring that the local models and function estimates are eventually sufficiently accurate almost surely; ASTRO-DF is a representative example of an algorithm built on this principle~\cite{MR3880261}. Subsequent work establishes almost-sure iteration and sample complexity results for ASTRO-DF and quantifies how variance-reduction devices, such as common random numbers (CRN), can improve sample complexity of these methods by replacing stringent per-point accuracy requirements with the difference accuracy requirements~\cite{MR4959972}. A related development 
uses similar difference error tail bounds within the supermartingle complexity analysis to obtain 
reduced sample sizes (with further benefits under CRN) while preserving convergence guarantees for both trust-region and direct-search methods~\cite{MR4758401}.\\
On the direct-search side, StoMADS extends MADS to stochastic black-box objectives through probabilistic estimates and variance control, and uses a mesh-based strategy to obtain almost-sure convergence to Clarke stationary points~\cite{MR4238147}, although with no complexity analysis. 
Constrained variants incorporate progressive barriers and probabilistic feasibility control~\cite{MR4550962}. For smooth objectives, the work by Dzahini~\cite{MR4359469} provides an expected-complexity analysis for stochastic direct search under a power-type sufficient-decrease condition, again via a supermartingale-based argument, and a recent survey summarizes current direct-search paradigms and guarantees, including stochastic settings and line-search variants~\cite{MR4919092}. Linesearch-based DFO with extrapolation under noisy oracles has also been studied recently, with convergence and complexity results in smooth regimes~\cite{desantis2025linesearchbasedderivativefreemethodnoisy}. Another zeroth-order paradigm is randomized smoothing, where one replaces the
objective by a smoothed surrogate and suitably build random estimators of the surrogate gradient. This idea is
central to the random gradient-free framework of Nesterov and
Spokoiny~\cite{MR3627456}. For nonsmooth nonconvex objectives, \cite{LinEtAl2022GFM} was among the first works to connect randomized smoothing
with Goldstein stationarity and to analyze gradient-free methods for reaching
\((r,\varepsilon)\)-Goldstein stationary points. This line was later sharpened
by Kornowski and Shamir~\cite{JMLR:kornowskishamir}, who obtained optimal
dimension-dependence for stochastic zero-order nonsmooth nonconvex optimization.
Our work fits into the stochastic direct-search/line-search thread, but targets a fully nonsmooth, derivative-free setting: we couple dense random directions with a run-until-failure line-search policy. We prove almost-sure convergence in the spirit of~\cite{MR4758401}, by relying on Clarke~\cite{Clarke1990} and Goldstein/Lebourg nonsmooth tools~\cite{Goldstein1977,Jahn1996}. Moreover, we derive an expected-complexity result by leveraging the supermartingale framework proposed in~\cite{MR4151319}, while working under minimal oracle assumptions.

%% file: sections/contributions.tex
The proposed approach combines a line search strategy with random directions uniformly generated on the unit sphere. 
Under conditional tail-accuracy assumptions (motivated by the weak tail-bound framework introduced in~\cite{MR4758401}), we prove almost-sure convergence to Clarke stationary points and, by suitably
adapting the analyses in~\cite{MR4151319,MR4060460}, we derive an expected
iteration complexity bound of order
\(\mathcal O(\max\{r^{-p},\varepsilon^{-p/(p-1)}\})\) for reaching an
\((r,\varepsilon)\)-Goldstein stationary point. We also obtain the corresponding
expected tested-point complexity bound
\(\mathcal O(\varepsilon^{1-n}\max\{r^{-p},\varepsilon^{-p/(p-1)}\})\), where
tested-point complexity counts the number of points at which stochastic function
estimates are queried.
The analysis relies exclusively on nonsmooth tools, that is Clarke and Goldstein subdifferentials and the Lebourg mean value theorem (see, e.g., \cite{Clarke1990,Goldstein1977,Jahn1996} for further details on these theoretical tools). To replace gradient-based descent estimates in the complexity argument, we test multiple \emph{dense random directions} per iteration (see, e.g., \ \cite{MR3627456}) and establish a geometric success guarantee: with probability strictly larger than $1/2$, at least one sampled direction falls in a spherical cap aligned with a Goldstein-descent direction. This spherical-cap mechanism is fundamentally different from the positive-spanning arguments used in smooth direct-search analyses (see, e.g., \cite{MR4359469}). 
To the best of our knowledge, this is the first convergence and expected-complexity analysis for an \emph{extrapolation-based direct-search} method in a  \emph{nonsmooth} setting.
A recent work by De Santis, Liuzzi and Lucidi~\cite{desantis2025linesearchbasedderivativefreemethodnoisy} studies a stochastic line-search scheme in a smooth setting, assuming continuously differentiable objectives with $L$-Lipschitz gradients and fixed-confidence accuracy for the function estimates; in that framework, the sampling strategy enforces a variance-stepsize coupling (of order $\Delta_k^4$) and the resulting complexity guarantees are stated in terms of the expected gradient norm (e.g., $\mathcal O(\varepsilon^{-2})$).
The present work addresses a complementary regime, focusing on merely locally Lipschitz, fully nonsmooth objectives and purely zeroth-order stochastic information. In contrast to analyses that assume a variance-stepsize coupling directly on the stochastic estimates, we impose conditional polynomial-tail accuracy conditions. 
 We combine a line search with dense random directions and obtain convergence and expected-complexity results in terms of Clarke/Goldstein stationarity. 

%% file: sections/algorithm.tex
\section{Algorithm: Direct Search with Extrapolation for Nonsmooth Objectives} \label{sec:algo}

We consider a direct-search method with extrapolation for possibly nonsmooth
stochastic zeroth-order optimization. The outer scheme is given in
Algorithm~\ref{alg:dse2-cartis}, and the stochastic line-search subroutine is given in
Algorithm~\ref{alg:lsep}. We fix $\Dexp\in (1,2]$, $\theta>0$, $\gamma\in(0,1)$, 
$\bar{m}\in\mathbb N$,  and $\Imax\in\mathbb N$.
Random
quantities are denoted by uppercase letters and their realizations by lowercase
letters. At the beginning of iteration $k$, the current iterate and stepsize are
$X_k\in\mathbb R^n$ and $\Delta_k>0$. The algorithm then samples
$\bar{m}$ candidate directions
$D_{k,1},\ldots,D_{k,\bar{m}}\in\mathbb S^{n-1},
$
where we denote by
\(
    \mathbb S^{n-1}
    :=
    \{d\in\mathbb R^n:\|d\|=1\}
\)
the unit sphere in \(\mathbb R^n\). We use two sigma-algebras. Let $\mathcal G_{k-1}$ denote the information
available before sampling the directions at iteration $k$, so that $X_k$ and
$\Delta_k$ are $\mathcal G_{k-1}$-measurable. Conditionally on
$\mathcal G_{k-1}$, the directions $D_{k,1},\ldots,D_{k,\bar{m}}$ are sampled
independently and uniformly on $\mathbb S^{n-1}$. After sampling them, but
before drawing any stochastic function estimates at iteration $k$, we set
\[
    \mathcal F_{k-1}
    :=
    \mathcal G_{k-1}
    \vee
    \sigma(D_{k,1},\ldots,D_{k,\bar{m}}).
\]
Thus $X_k$, $\Delta_k$, and all directions $D_{k,m}$ are
$\mathcal F_{k-1}$-measurable, while the estimates generated during iteration
$k$ are not. Oracle assumptions are stated conditionally on
$\mathcal F_{k-1}$. For $m=1,\ldots,\bar{m}$ and $i\ge0$, define
\[
    \Delta_{k,i}:=\gamma^{-i}\Delta_k,
    \qquad
    X_{k,m}^{(i)}:=X_k+\Delta_{k,i}D_{k,m}.
\]
We also set $X_k^{(-1)}:=X_k$ and $\Delta_{k,-1}:=0$. The corresponding true and
estimated values are
\[
    f_{k,m}^{(i)}:=f(X_{k,m}^{(i)}),
    \qquad
    \tilde f_{k,m}^{(i)}:=\tilde f(X_{k,m}^{(i)}),
    \qquad i\ge0,
\]
and
\[
    f_k^{(-1)}:=f(X_k),
    \qquad
    \tilde f_k^{(-1)}:=\tilde f(X_k).
\]
At iteration $k$, the algorithm first samples $\bar{m}$ search directions independently, then it observes $\tilde f_k^{(-1)}$ and finally
tests the directions $D_{k,1},\ldots,D_{k,\bar{m}}$ sequentially. Direction $m$ is
successful at depth $i\ge0$ if
\[
    \tilde f_k^{(-1)}-\tilde f_{k,m}^{(i)}
    \ge
    \theta\Delta_{k,i}^{\Dexp}.
\]
The algorithm accepts the first direction for which the line search succeeds.
We denote its index by $m_k^*$ and the last successful extrapolation depth by
$H_k\ge0$. If no direction succeeds, we set $H_k=-1$. Thus, on a successful iteration,
\[
    D_k:=D_{k,m_k^*},
    \qquad
    X_k^{(i)}:=X_{k,m_k^*}^{(i)},
    \qquad
    f_k^{(i)}:=f_{k,m_k^*}^{(i)},
    \qquad
    \tilde f_k^{(i)}:=\tilde f_{k,m_k^*}^{(i)}.
\]
The accepted step length is
\[
    \Delta_k^{\rm acc}
    :=
    \begin{cases}
        0, & H_k=-1,\\[1mm]
        \gamma^{-H_k}\Delta_k, & H_k\ge0.
    \end{cases}
\]
Hence $X_{k+1}=X_k+\Delta_k^{\rm acc}D_k,$
with the convention that $X_{k+1}=X_k$ when $H_k=-1$. The stepsize update is
\[
\Delta_{k+1}
=
\begin{cases}
    \gamma\Delta_k, & H_k=-1,\\[1mm]
    \gamma^{-1}\Delta_k, & H_k=0,\\[1mm]
    \gamma^{-H_k}\Delta_k, & H_k\ge1.
\end{cases}
\] 
Equivalently, on successful iterations,
$\Delta_{k+1}
    =
    \max\{\Delta_k^{\rm acc},\gamma^{-1}\Delta_k\}.$
Thus every unsuccessful iteration contracts the radius by $\gamma$, whereas
every successful iteration expands the next radius by at least $\gamma^{-1}$.
\begin{algorithm}[H]
\caption{DSE}
\label{alg:dse2-cartis}
\begin{algorithmic}[1]
\State \textbf{Inputs:} $x_0\in\R^n$, $\delta_0>0$, $\theta>0$, $\gamma\in(0,1)$, $\bar{m}\in\N$, $\Imax \in \N$
\For{$k=0,1,2,\dots$}
  \State Sample $d_{k,1}, \ \dots, \ d_{k,\bar{m}}\in \mathbb S^{n-1}$ independently
  \State Observe the baseline estimate $\tilde f_k^{(-1)}\gets \tilde f(x_k)$
  \For{$m=1$ \textbf{to} $\bar{m}$}
    \State $(\delta^{\mathrm{acc}},\,\delta^{\mathrm{next}})
      \gets \textbf{StochasticLinesearch}\left(x_k,\delta_k,d_{k,m},\theta,\gamma, \f_k^{(-1)}, \Imax\right)$
    \If{$\delta^{\mathrm{acc}}>0$}
   
      \State \textbf{break}
    \EndIf
  \EndFor
     \State $d_k=d_{k,m}$
    \State $ \delta_{k+1}\gets \delta^{\mathrm{next}}$
  \State $x_{k+1}\gets x_k+\delta^{\mathrm{acc}}\,d_{k}$
\EndFor
\end{algorithmic}
\end{algorithm}

\begin{algorithm}[H]
\caption{StochasticLinesearch}
\label{alg:lsep}
\begin{algorithmic}[1]
\State \textbf{Input:} $x \in \mathbb{R}^n$, $\delta > 0$, $d \in \mathbb{R}^n$, $\theta > 0$, $\gamma \in (0,1)$, $\f_k^{(-1)}, \Imax$
        \State Set $\delta_0 = \delta$, $\delta_i=\gamma^{-i}\delta$, \quad $i=0,\ldots,\Imax$
\If{$\tilde f(x+\delta_0 d)>\tilde f_k^{(-1)}-\theta\delta_0^p$}
    \State \Return $(0,\gamma\delta_0)$
\EndIf
\State $h=\max\left\{j\in[0:\Imax]:
\tilde f(x+\delta_i d)\le \tilde f_k^{(-1)}-\theta\delta_i^p, \forall\ i\in[0:j] \right\}$
\If{$h=0$}
    \State \Return $(\delta_0,\delta_0/\gamma)$
\Else
    \State \Return $(\delta_h,\delta_h)$
\EndIf
\end{algorithmic}
\end{algorithm}

\noindent
We  would like to highlight that the maximization problem defined in Step 6 of the stochastic linesearch procedure is only a compact way of describing a sequential run-until-failure procedure.
In practice, after the level \(i=0\) test succeeds, the line search evaluates the
levels \(i=1,2,\ldots,\Imax\) one at a time and stops as soon as the
sufficient-decrease test fails. The returned value of \(h\) is the last
consecutive level for which all tests from \(0\) up to \(h\) have succeeded.
Thus, the algorithm does not evaluate extrapolation levels beyond the first
failed test.  The following remark finally clarifies the role of the maximum extrapolation depth
\(\Imax\) in Algorithm~\ref{alg:lsep}.

\begin{remark}[Role of the extrapolation cap]
The extrapolation cap \(\Imax\) is mainly an analytical device and is not
restrictive in practice. Indeed, any implementation is run under a finite
oracle-call budget, and \(\Imax\) can be chosen as large as this
budget allows. Its role in the analysis is to keep the extrapolation search
tree uniformly finite, which yields a simple maximal-error bound and a direct
tested-point complexity estimate (See Sections \ref{sec:assumptions}-\ref{sec:clarke complexity}).
\end{remark}


%% file: sections/assumptions_and_preliminaries.tex
\section{Assumptions and Preliminaries}\label{sec:assumptions}
We now state the main assumptions on the objective and the stochastic oracle, deriving the key error bounds needed for the analysis. Throughout, we condition on $\F_{k-1}$, which contains $X_k, \D_k$, and the sampled directions $D_{k,1}, \ \dots, \ D_{k,\bar{m}}$, but not the stochastic estimates generated at iteration $k$. We assume noisy function estimates with the following basic structure.
\begin{assumption}[Compactness of the iterates]\label{ass:compact}
There exists a nonempty compact set \( {\mathcal C}\subset\mathbb R^n\) such that $X_k\in {\mathcal C} \ \text{a.s. for all $k\ge0$}.$
Moreover, the initial stepsize $\Delta_0>0$ is deterministic and finite.
\end{assumption}

\begin{remark}[Compactness of the tested points]
\label{rem:compact-tested-points}
Under the fixed maximum extrapolation depth
\(\Imax\in\mathbb N_0\),
and Assumption~\ref{ass:compact}, all tested points lie in a compact
enlargement of \({\mathcal C}\). Indeed, let
\[
    \Delta_{\mathcal C}:=\operatorname{diam}({\mathcal C}),
    \qquad
    \bar\Delta:=\max\{\Delta_0,\gamma^{-1}\Delta_{\mathcal C}\}.
\]
We first prove that $\Delta_k\le \bar\Delta,
    \ \forall k\ge0.$
The claim is true for \(k=0\) by the definition of \(\bar\Delta\). Suppose that
\(\Delta_k\le\bar\Delta\). If iteration \(k\) is unsuccessful, then
\[
    \Delta_{k+1}=\gamma\Delta_k\le\Delta_k\le\bar\Delta.
\]
If iteration \(k\) is successful with \(H_k=0\), then 
$  X_{k+1}=X_k+\Delta_kD_k$ is the accepted point.
Since \(X_k,X_{k+1}\in {\mathcal C}\), we have
\[
    \Delta_k=\|X_{k+1}-X_k\|\le \Delta_{\mathcal C}.
\]
Therefore,
\[
    \Delta_{k+1}=\gamma^{-1}\Delta_k\le\gamma^{-1}\Delta_{\mathcal C}\le\bar\Delta.
\]
Finally, if iteration \(k\) is successful with \(H_k\ge1\), then $X_{k+1}=X_k+\gamma^{-H_k}\Delta_kD_k,$ and hence
\[
    \Delta_{k+1}
    =
    \gamma^{-H_k}\Delta_k
    =
    \|X_{k+1}-X_k\|
    \le
    \Delta_{\mathcal C}
    \le
    \bar\Delta.
\]  
Thus, by induction, $\Delta_k\le\bar\Delta, \ \forall k\ge0.$
Consequently, for every tested point,
\[
    \|X_{k,m}^{(i)}-X_k\|
    =
    \gamma^{-i}\Delta_k
    \le
    \gamma^{-\Imax}\bar\Delta.
\]
Therefore
\[
    X_{k,m}^{(i)}
    \in
    {{\mathcal C}_{enl}}
    :=
    {\mathcal C}+\gamma^{-\Imax}\bar\Delta\,\mathbb B,
    \qquad
    1\le m\le \bar{m},\quad 0\le i\le \Imax.
\]
Here,
\(
    \mathbb B:=\{z\in\mathbb R^n:\|z\|\le 1\}
\)
denotes the closed unit ball in \(\mathbb R^n\), and the symbol \(+\) denotes
the Minkowski sum of sets, namely
\(
    A+B:=\{a+b:a\in A,\ b\in B\}.
\)
Since \({\mathcal C}\) is compact and \(\gamma^{-\Imax}\bar\Delta\,\mathbb B\) is
compact, the set  ${\mathcal C}_{enl}$ is compact. Hence, all iterates and all tested trial
points generated by the algorithm lie in the compact set ${\mathcal C}_{enl}$.
\end{remark}

\begin{assumption}[Conditional independence]\label{ass:indep}
At iteration \(k\), conditionally on \(\mathcal F_{k-1}\), the stochastic
function estimates associated with the base point and with the search
tree,
\[
    \tilde f_k^{(-1)}
    \quad\text{and}\quad
    \left\{
        \tilde f_{k,m}^{(i)}
        :
        1\le m\le \bar{m},\ 0\le i\le \Imax
    \right\},
\]
are mutually independent.
\end{assumption}
\begin{assumption}[$\Dexp$-tail for single-point errors]\label{ass:quadratic tail}
There exists $\varepsilon_h>0$ such that, for every $\alpha\ge \varepsilon_h$:
\begin{equation}\label{eq:A4}
\begin{aligned}
&\mathbb P\!\left(
|\tilde f_{k}^{(-1)}-f_{k}^{(-1)}|
\ge
\alpha\,\Delta_k^{\Dexp}
\ \middle|\ 
\F_{k-1}
\right)
\le
\frac{\varepsilon_h}{\alpha^{\aexp}},
\\[1mm]
&\mathbb P\!\left(
|\tilde f_{k,m}^{(i)}-f_{k,m}^{(i)}|
\ge
\alpha\,\Delta_k^{\Dexp}
\ \middle|\ 
\F_{k-1}
\right)
\le
\frac{\varepsilon_h}{\alpha^{\aexp}},
\quad
i\in[0:\Imax],\ m\in[1:\bar{m}].
\end{aligned}
\end{equation}

\end{assumption}

\noindent
To simplify the notation, from now on we set $\bar E_{k,m}^{(i)}:=\tilde f^{(i)}_{k,m} - f^{(i)}_{k,m}$, with
\(
    \bar E_k^{(-1)}
    :=
    \tilde f_k^{(-1)}-f_k^{(-1)}
\),
 and define the error difference estimator \[E_{k,m}^{(i)}:=(\tilde f_k^{(-1)}-\tilde f_k^{(i)})-(f_k^{(-1)}-f_k^{(i)}) = \bar E_{k}^{(-1)} - \bar E_{k,m}^{(i)},\]
Assumptions~\ref{ass:compact} and~\ref{ass:indep} are standard in stochastic derivative-free optimization (see, e.g., \cite{MR4238147,MR4758401}). 
In particular, compactness of the iterate set allows us to invoke tools from nonsmooth analysis
(such as the Clarke subdifferential) and ensures, under mild regularity (e.g., continuity on ${\mathcal C}_{enl}$), that
the objective $f$ attains its minimum. 
For notational simplicity, since our analysis will be restricted to a compact set, we will often write $f$ is $L$-Lipschitz.  Assumptions~\ref{ass:quadratic tail}, instead, is a technical condition that provides
quantitative control of the stochastic estimation error $\tilde f$ and
is needed to establish the merit-function drift bounds used throughout the analysis. From this basic set of hypotheses we derive several consequences that will be used in the convergence
and complexity proofs. For readability, we list the statements here and defer the proofs to the Appendix \ref{appendix:A}.
Similarly, from Assumption \ref{ass:quadratic tail} we can get several other useful results by simply integrating the tails.
\begin{lemma}[Conditional Mean Bound]
\label{lem: conditional mean bound}
Let Assumption \ref{ass:quadratic tail} hold. For each $k$, $i\in[-1: \Imax]$ and $m\in[1:\bar{m}]$ recall the estimation error $\bar E_{k,m}^{(i)}:=\tilde f^{(i)}_{k,m} - f^{(i)}_{k,m}$. Then, there exists a constant $\mu_h>0$ such that the following bound on the conditional mean holds, uniformly in $k$ and $i$
\[
\mathbb{E}\!\left[\,|\bar E_{k,m}^{(i)}|\ \middle|\ \mathcal F_{k-1}\right]
\ \le\ \mu_h\,\Delta_k^{\Dexp},
\]
with $\mu_h$ a suitably chosen positive parameter.
\end{lemma}
From the previous lemma we can also bound the conditional expectation of the difference estimator $E_{k,m}^{(i)}$.
\begin{lemma}[Conditional $L^1$ bound for the error difference estimator]
\label{lem:diff-L1-from-A44}
Under Assumption \ref{ass:quadratic tail}, 
for every outer iteration $k$ and every inner index $i\in\{-1,0,1,\dots, \Imax\}$,
\begin{equation}\label{eq:diff-L1-core}
\mathbb E\!\left[\,|E_{k,m}^{(i)}|\,\middle|\,\mathcal F_{k-1}\right]
\ \le\ 2\,\mu_h\,\Delta_k^{\Dexp} \qquad \text{a.s.},
\end{equation}
with $\mu_h$ as in Lemma \ref{lem: conditional mean bound}. Consequently, the signed conditional mean is also controlled:
\begin{equation}\label{eq:diff-signed}
\big|\mathbb E\!\left[\,E_{k,m}^{(i)}\,\middle|\,\mathcal F_{k-1}\right]\big|
\ \le\ 2\,\mu_h\,\Delta_k^{\Dexp} \qquad \text{a.s.}.
\end{equation}
\end{lemma}
\begin{lemma}[Bounded maximal iteration error]
\label{ass:postL1}
At iteration \(k\), after sampling the trial directions
\(D_{k,1},\ldots,D_{k,\bar{m}}\), 
define the maximal difference-estimation error over the potential search tree by
\[
E_k^{\max}
:=
\max_{1\le m\le \bar{m}}
\max_{0\le i\le \Imax}
\left|
E_{k,m}^{(i)}
\right|.
\]
If Assumption \ref{ass:quadratic tail} holds, then via Lemma \ref{lem:diff-L1-from-A44}, one has
\[
\mathbb E\left[
E_k^{\max}
\mid
\mathcal F_{k-1}
\right]
\le
c_{\max}\Delta_k^p
\qquad
\text{a.s.},
\]
where \(c_{\max} = 2\mu_h\bar{m}(\Imax+1)\).
\end{lemma}
\noindent

\noindent
Finally, we note that the conditional \(\Dexp\)-tail bound in
Assumption~\ref{ass:quadratic tail} directly implies the commonly used notion of
\(\beta\)-probabilistically \(\varepsilon_f\)-accurate function estimates, as
made precise in the following remark.
\begin{remark}[From $\Dexp$-tails to $\beta$-probabilistic $\varepsilon_f$-accuracy]
\label{rem:A4-to-BPE}
Assumption~\ref{ass:quadratic tail} states that, for some $\varepsilon_h>0$ and all
$i\in\{-1,0,1,\dots, \Imax\}$,
\[
\mathbb P\!\left(|\tilde f_k^{(i)}-f_k^{(i)}|\ \ge\ \alpha\,\Delta_k^{\Dexp}\ \middle|\ \mathcal F_{k-1}\right)
\ \le\ \frac{\varepsilon_h}{\alpha^{\aexp}}\qquad\text{for all }\alpha\ge \varepsilon_h.
\]
Fix a target confidence $\beta\in(0,1)$. Choose
\[
\varepsilon_f\ \ge\ \max\!\left\{\varepsilon_h,\ \left(\frac{\varepsilon_h}{\,1-\beta\,}\right)^{\frac{\Dexp-1}{\Dexp}}\right\}.
\]
Then, setting $\alpha=\varepsilon_f$ in the tail bound gives
\[
\mathbb P\!\left(|\tilde f_k^{(i)}-f_k^{(i)}|\ \le\ \varepsilon_f\,\Delta_k^{\Dexp}\ \middle|\ \mathcal F_{k-1}\right)
\ \ge\ 1-\frac{\varepsilon_h}{\varepsilon_f^{\aexp}}\ \ge\ \beta,
\]
i.e., the usual $\beta$-probabilistically $\varepsilon_f$-accurate condition holds uniformly in $i$. If the tail in Assumption~\ref{ass:quadratic tail} holds for all $\alpha>0$, the simpler choice
$\varepsilon_f=(\varepsilon_h/(1-\beta))^{\frac{\Dexp-1}{\Dexp}}$ suffices.
\end{remark}
We conclude by observing that Assumption~\ref{ass:quadratic tail} is not restrictive in practice. Indeed, standard sampling strategies in stochastic derivative-free optimization often rely on averaging multiple independent realizations of the stochastic oracle; see, for example,~\cite{MR4238147,MR3800867,MR3880261}. Under the classical assumptions commonly imposed on such sample-average estimators, one obtains the stronger  conditions typically adopted in the stochastic derivative-free optimization literature, which in turn imply Assumption~\ref{ass:quadratic tail}. For completeness, Appendix~\ref{appendix:random estimates} establishes this connection for a standard sample-average estimator and derives the corresponding batch-size requirements.

%% file: sections/convergence.tex
\section{Main Convergence Result} \label{sec:convergence}

Following the trust-region/direct-search literature (see,e.g., \cite{MR3506227,MR4758401}), we analyze progress through the \emph{merit function}
\begin{equation}\label{eq:phi-def}
  \Phi_k \;=\; f(X_k)-f_{min} \;+\; \eta\,\Delta_k^{\Dexp},
\end{equation}
where \(f_{\min}\) is a lower bound for \(f\) on the compact region containing the
iterates, and $\eta$ is a positive parameter. The term \(f(X_k)-f^\star\) measures objective decrease, while
\(\eta\,\Delta_k^{\Dexp}\) penalizes overly aggressive step-size growth during successful extrapolations. This choice ensures that the \emph{merit-function
drift}
\(\EE[\Phi_k-\Phi_{k+1}\mid\F_{k-1}]\) trades off true decrease against step-size changes in a way
that yields a uniform lower bound proportional to \(\Delta_k^{\Dexp}\).
At iteration \(k\), define
$H_k\in\{-1,0,\ldots,\Imax\},$ 
where \(H_k=-1\) denotes failure of the whole iteration, while \(H_k=i\ge0\)
denotes success at extrapolation depth \(i\). For \(i=0,\ldots,\Imax\), set
\[
    \Delta_{k,i}:=\gamma^{-i}\Delta_k,
    \qquad
    \pi_i:=\mathbb P(H_k=i\mid\mathcal F_{k-1}),
\]
and also
\[
    \pi_{-1}:=\mathbb P(H_k=-1\mid\mathcal F_{k-1}).
\]
Thus
\[
    \pi_{-1}+\sum_{i=0}^{\Imax}\pi_i=1.
\]
On a successful iteration, let \(m_k^*\) be the accepted direction index, so that
\[
    X_{k+1}
    =
    X_{k,m_k^*}^{(H_k)}
    =
    X_k+\Delta_{k,H_k}D_{k,m_k^*}.
\]
On a failed iteration, \(X_{k+1}=X_k\). The stepsize update is
\[
\Delta_{k+1}
=
\begin{cases}
    \gamma\Delta_k, & H_k=-1,\\[1mm]
    \gamma^{-1}\Delta_k, & H_k=0,\\[1mm]
    \gamma^{-i}\Delta_k, & H_k=i\ge1.
\end{cases}
\]
Recall the maximal difference-estimation error
\[
E_k^{\max}
:=
\max_{1\le m\le \bar{m}}
\max_{0\le i\le \Imax}
\left|
\left(
\tilde f_k^{(-1)}-\tilde f_{k,m}^{(i)}
\right)
-
\left(
f_k^{(-1)}-f_{k,m}^{(i)}
\right)
\right|.
\]

\begin{lemma}[Merit-function drift]
\label{lem:phi_drift_main}
Let Assumption~\ref{ass:quadratic tail} hold. Then, by
Lemma~\ref{ass:postL1}, there exists \(c_{\max}>0\) such that
\[
    \mathbb E\!\left[
        E_k^{\max}
        \,\middle|\,
        \mathcal F_{k-1}
    \right]
    \le
    c_{\max}\Delta_k^\Dexp
    \qquad \text{a.s.}
\]
Assume that
\begin{equation}
\label{ass:eta}
    \eta>
    \frac{c_{\max}}{1-\gamma^\Dexp},
\end{equation}
and
\begin{equation}
\label{ass:theta_eta_compatibility}
    \theta
    \ge
    \eta\max\{1,\gamma^{-\Dexp}-\gamma^\Dexp\}.
\end{equation}
Then
\begin{equation}
\label{eq:drift_uniform}
\mathbb E\!\left[
    \Phi_k-\Phi_{k+1}
    \,\middle|\,
    \mathcal F_{k-1}
\right]
\ge
\left(
    \eta(1-\gamma^\Dexp)-c_{\max}
\right)\Delta_k^\Dexp
>0
\qquad \textup{a.s.}
\end{equation}
\end{lemma}

\begin{proof}
All expectations are taken conditionally on $\mathcal F_{k-1}$. We decompose
\[
    \Phi_k-\Phi_{k+1}
    =
    f(X_k)-f(X_{k+1})
    +
    \eta(\Delta_k^\Dexp-\Delta_{k+1}^\Dexp).
\]

\medskip
\noindent
\textbf{Step 1: True-reduction term.}
Fix $i\in[0:\Imax]$. Define
\[
E_{k,m_k^*}^{(i)}
:=
\bigl(\tilde f_k-\tilde f_{k,m_k^*}^{(i)}\bigr)
-
\bigl(f_k-f_{k,m_k^*}^{(i)}\bigr).
\]
Then 
\[
f_k-f_{k,m_k^*}^{(i)}
=
\bigl(\tilde f_k-\tilde f_{k,m_k^*}^{(i)}\bigr)
-
E_{k,m_k^*}^{(i)}.
\]
On $\{H_k=i\}$, the sufficient decrease holds, so
\[
\tilde f_k-\tilde f_{k,m_k^*}^{(i)}
\geq
\theta\Delta_{k,i}^{\Dexp}
.
\]
Therefore
\[
\begin{aligned}
&\mathbf 1_{\{H_k=i\}}
\bigl(f_k-f_{k,m_k^*}^{(i)}\bigr)
\geq
\mathbf 1_{\{H_k=i\}}
\Bigl(
    \theta\Delta_{k,i}^{\Dexp}
    -
    E_{k,m_k^*}^{(i)}
\Bigr).
\end{aligned}
\]
Taking conditional expectations and using
\[
    \Delta_{k,i}^{\Dexp}
    =
    \gamma^{-i\Dexp }\Delta_k^\Dexp,
    \qquad
    \mathbb E[\mathbf 1_{\{H_k=i\}}\mid\mathcal F_{k-1}]
    =
    \pi_i,
\]
we obtain
\[
\begin{aligned}
\mathbb E\!\left[
    \mathbf 1_{\{H_k=i\}}
    \bigl(f_k-f_{k,m_k^*}^{(i)}\bigr)
    \,\middle|\,
    \mathcal F_{k-1}
\right]
\geq
\pi_i\theta\Delta_{k,i}^{\Dexp}
-
\mathbb E\!\left[
    \mathbf 1_{\{H_k=i\}}E_{k,m_k^*}^{(i)}
    \,\middle|\,
    \mathcal F_{k-1}
\right].
\end{aligned}
\]
Since $|E_{k,m_k^*}^{(i)}|\le E_k^{\max}$, it follows that 
\[
\begin{aligned}
&\mathbb E\!\left[
    \mathbf 1_{\{H_k=i\}}
    \bigl(f_k-f_{k,m_k^*}^{(i)}\bigr)
    \,\middle|\,
    \mathcal F_{k-1}
\right]
\ge
\pi_i\theta\Delta_{k,i}^{\Dexp}
-
\mathbb E\!\left[
    \mathbf 1_{\{H_k=i\}}E_k^{\max}
    \,\middle|\,
    \mathcal F_{k-1}
\right].
\end{aligned}
\]
Summing over $i\in[0:\Imax]$ gives
\[
\begin{aligned}
&\sum_{i=0}^{\Imax}
\mathbb E\!\left[
    \mathbf 1_{\{H_k=i\}}
    \bigl(f_k-f_{k,m_k^*}^{(i)}\bigr)
    \,\middle|\,
    \mathcal F_{k-1}
\right]
\ge
\sum_{i=0}^{\Imax}\pi_i\theta\Delta_{k,i}^{\Dexp}
-
\mathbb E\!\left[
    \mathbf 1_{\{H_k\ge0\}}E_k^{\max}
    \,\middle|\,
    \mathcal F_{k-1}
\right].
\end{aligned}
\]
By Lemma~\ref{ass:postL1},
\[
\mathbb E\!\left[
    \mathbf 1_{\{H_k\ge0\}}E_k^{\max}
    \,\middle|\,
    \mathcal F_{k-1}
\right]
\le
\mathbb E[E_k^{\max}\mid\mathcal F_{k-1}]
\le
c_{\max}\Delta_k^\Dexp.
\]
Hence
\begin{equation}\label{eq:true_red_explicit_2}
\sum_{i=0}^{\Imax}
\mathbb E\!\left[
    \mathbf 1_{\{H_k=i\}}
    \bigl(f_k-f_{k,m_k^*}^{(i)}\bigr)
    \,\middle|\,
    \mathcal F_{k-1}
\right]
\ge
\left(
    \sum_{i=0}^{\Imax}\pi_i\theta\gamma^{-\Dexp i}
    -
    c_{\max}
\right)
\Delta_k^\Dexp .
\end{equation}

\medskip
\noindent
\textbf{Step 2: Stepsize term.}
By the modified stepsize update,
\[
\Delta_{k+1}
=
\mathbf 1_{\{H_k=-1\}}(\gamma\Delta_k)
+
\mathbf 1_{\{H_k=0\}}(\gamma^{-1}\Delta_k)
+
\sum_{i=1}^{\Imax}
\mathbf 1_{\{H_k=i\}}(\gamma^{-i}\Delta_k).
\]
Therefore
\[
\begin{aligned}
\eta(\Delta_k^\Dexp-\Delta_{k+1}^\Dexp)
={}&
\eta\Delta_k^\Dexp
\Bigg[
    \mathbf 1_{\{H_k=-1\}}(1-\gamma^\Dexp)
    +
    \mathbf 1_{\{H_k=0\}}(1-\gamma^{-\Dexp})+
\\
&\qquad\qquad
    +
    \sum_{i=1}^{\Imax}
    \mathbf 1_{\{H_k=i\}}
    \bigl(1-\gamma^{-\Dexp i}\bigr)
\Bigg].
\end{aligned}
\]
Taking conditional expectations yields
\begin{equation}\label{eq:stepsize}
\begin{aligned}
\mathbb E\!\left[
    \eta(\Delta_k^\Dexp-\Delta_{k+1}^\Dexp)
    \,\middle|\,
    \mathcal F_{k-1}
\right]
={}&
\eta\Delta_k^\Dexp
\Bigg[
    \pi_{-1}(1-\gamma^\Dexp)
    +
    \pi_0(1-\gamma^{-\Dexp})+
\\
&\qquad\qquad
    +
    \sum_{i=1}^{\Imax}
    \pi_i(1-\gamma^{-\Dexp i})
\Bigg].
\end{aligned}
\end{equation}

\medskip
\noindent
\textbf{Step 3: Combining the two estimates.}
Combining~\eqref{eq:true_red_explicit_2} and~\eqref{eq:stepsize}, and recalling that the
true-reduction term is zero on $\{H_k=-1\}$, we obtain

\begin{equation}\label{eq:drift_full}
\begin{aligned}
\mathbb E\!\left[\Phi_k-\Phi_{k+1}\,\middle|\,\mathcal F_{k-1}\right]
&\ge \Bigg[\pi_{-1}\eta(1-\gamma^\Dexp)
+\pi_0\Bigl(\theta+\eta(1-\gamma^{-\Dexp})\Bigr)+ \\
&\qquad
+\sum_{i=1}^{\Imax}\pi_i
\Bigl(\theta\gamma^{-\Dexp i}
+\eta(1-\gamma^{-\Dexp i})\Bigr)
-c_{\max}\Bigg]\Delta_k^\Dexp .
\end{aligned}
\end{equation}
Define
\[
    A:=\eta(1-\gamma^\Dexp), \quad 
    B:=\theta+\eta(1-\gamma^{-\Dexp})
      =\theta-\eta(\gamma^{-\Dexp}-1),
\]
and, for $i\ge1$,
\[
    C_i
    :=
    \theta\gamma^{-\Dexp i}
    +
    \eta(1-\gamma^{-\Dexp i})
    =
    \eta+(\theta-\eta)\gamma^{-\Dexp i}.
\]
Then, under~\eqref{ass:theta_eta_compatibility},
\[
    B\ge A,
    \qquad
    C_i\ge A\quad \forall i\ge1.
\]
\noindent
Indeed, $B\ge A$ is equivalent to
\[
    \theta-\eta(\gamma^{-\Dexp}-1)
    \ge
    \eta(1-\gamma^\Dexp),
\]
that is,
\[
    \theta\ge \eta(\gamma^{-\Dexp}-\gamma^\Dexp),
\]
which follows from~\eqref{ass:theta_eta_compatibility}. Moreover,
\eqref{ass:theta_eta_compatibility} also implies $\theta\ge\eta$. Hence, for $i\ge1$,
\[
    C_i
    =
    \eta+(\theta-\eta)\gamma^{-\Dexp i}
    \ge
    \eta
    \ge
    \eta(1-\gamma^\Dexp)
    =
    A.
\]
Therefore
\[
    \pi_{-1}A+\pi_0B+\sum_{i=1}^{\Imax}\pi_iC_i
    \ge
    \left(\pi_{-1}+\pi_0+\sum_{i=1}^{\Imax}\pi_i\right)A
    =
    A.
\]
Using this in~\eqref{eq:drift_full}, we obtain
\[
\mathbb E\!\left[
    \Phi_k-\Phi_{k+1}
    \,\middle|\,
    \mathcal F_{k-1}
\right]
\ge
\left(
    A-c_{\max}
\right)\Delta_k^\Dexp.
\]
Substituting $A=\eta(1-\gamma^\Dexp)$ gives
\[
\mathbb E\!\left[
    \Phi_k-\Phi_{k+1}
    \,\middle|\,
    \mathcal F_{k-1}
\right]
\ge
\left(
    \eta(1-\gamma^\Dexp)-c_{\max}
\right)\Delta_k^\Dexp.
\]
By~\eqref{ass:eta}, the coefficient is strictly positive. This concludes the proof.
\end{proof}

\begin{lemma}\label{l:radius0}
Under the hypotheses of Lemma~\ref{lem:phi_drift_main},
it holds
\[
    \sum_{k=0}^{\infty}\Delta_k^{\Dexp}<\infty
    \qquad \text{a.s.}
\]
In particular,
\[
    \Delta_k\to0
    \qquad \text{a.s. as } k\to\infty .
\]
\end{lemma}

\begin{proof}
Define \(\Theta := \eta(1-\gamma^{\Dexp})-c_{\max}, \;\Theta>0\). Taking total expectations in the drift inequality gives
\[
    \mathbb E[\Phi_k-\Phi_{k+1}]
    \ge
    \Theta\,\mathbb E[\Delta_k^{\Dexp}]
    \qquad \forall k\in\mathbb N_0 .
\]
Summing from $k=0$ to $N$ yields
\[
    \Theta\sum_{k=0}^{N}\mathbb E[\Delta_k^{\Dexp}]
    \le
    \sum_{k=0}^{N}\mathbb E[\Phi_k-\Phi_{k+1}]
    =
    \mathbb E[\Phi_0-\Phi_{N+1}] .
\]
Since
\[
    \Phi_{N+1}
    =
    f_{N+1}-f_{\min}
    +
    \eta\Delta_{N+1}^{\Dexp}
    \ge 0,
\]
we obtain
\[
    \Theta\sum_{k=0}^{N}\mathbb E[\Delta_k^{\Dexp}]
    \le
    \mathbb E[\Phi_0].
\]
Therefore,
\[
    \sum_{k=0}^{N}\mathbb E[\Delta_k^{\Dexp}]
    \le
    \frac{\mathbb E[\Phi_0]}{\Theta}
    \qquad \forall N\in\mathbb N_0 .
\]
Letting $N\to\infty$ and using monotone convergence,
\[
    \mathbb E\!\left[
        \sum_{k=0}^{\infty}\Delta_k^{\Dexp}
    \right]
    =
    \sum_{k=0}^{\infty}\mathbb E[\Delta_k^{\Dexp}]
    \le
    \frac{\mathbb E[\Phi_0]}{\Theta}
    <
    \infty .
\]
Since the random variable
\(
    \sum_{k=0}^{\infty}\Delta_k^{\Dexp}
\)
is nonnegative and has finite expectation, it is finite almost surely. Hence
\(
    \sum_{k=0}^{\infty}\Delta_k^{\Dexp}<\infty
\)
almost surely.
Finally, since $\Delta_k^{\Dexp}\ge0$ and $p>1$, summability implies
\(
    \Delta_k^{\Dexp}\to0
    \ \ \text{a.s.}
\)
and therefore
\(
    \Delta_k\to0
    \ \ \text{a.s.}
\)
because $\Dexp>0$.
\end{proof}
\noindent
Thus, from the merit-function drift we have
\(\sum_{k=0}^\infty \Delta_k^{\Dexp}<\infty\) a.s., hence \(\Delta_k\to0\) along the sequence almost surely. This in turn guarantees the existence of a subsequence of unsuccessful iterations.
\begin{lemma}[Existence of infinitely many unsuccessful iterations] \label{lem:inf-unsuccessful}
Let
\[
    \mathcal U
    :=
    \{k\in\mathbb N_0 : \text{iteration } k \text{ is unsuccessful}\}
\]
be the random set of indices of unsuccessful iterations. Assume that
\[
    \sum_{k=0}^{\infty}\Delta_k^{\Dexp}<\infty
    \qquad \text{a.s.}
\]
Then, with probability one, the set $\mathcal U$ is infinite.
\end{lemma}

\begin{proof}
Argue by contradiction. Suppose that, with strictly positive probability,
there are only finitely many unsuccessful iterations. On this event, there
exists a finite random index $k^u$ such that every iteration $k\ge k^u$ is
successful. By the stepsize update rule, on a successful iteration we have
\[
\Delta_{k+1}
=
\begin{cases}
    \gamma^{-1}\Delta_k, & H_k=0,\\[1mm]
    \gamma^{-H_k}\Delta_k, & H_k\ge1.
\end{cases}
\]
Since $0<\gamma<1$, it follows in both cases that
\[
    \Delta_{k+1}\ge \gamma^{-1}\Delta_k>\Delta_k .
\]
In particular, on the event under consideration, the tail sequence
$\{\Delta_k\}_{k\ge k^u}$ is strictly increasing. Since $\Delta_{k^u}>0$, we get
\[
    \Delta_k\ge \Delta_{k^u}>0
    \qquad \forall k\ge k^u .
\]
Therefore
\[
    \sum_{k=k^u}^{\infty}\Delta_k^{\Dexp}
    \ge
    \sum_{k=k^u}^{\infty}\Delta_{k^u}^{\Dexp}
    =
    \infty,
\]
which contradicts the assumption
\[
    \sum_{k=0}^{\infty}\Delta_k^{\Dexp}<\infty
    \qquad \text{a.s.}
\]
Hence the probability that $\mathcal U$ is finite must be zero. Therefore
$\mathcal U$ is infinite almost surely.
\end{proof}
\noindent
It is on these unsuccessful iterations that our convergence analysis will concentrate.\\
We say that a subsequence $\{x_k\}_{k \in L}$ is refining if it converges to $x^*$ and if $\{d_k\}_{k \in L}$ is dense in the unit sphere. We will show that under \eqref{ass:eta}, \eqref{ass:theta_eta_compatibility}, if Assumptions  \ref{ass:compact}, \ref{ass:indep}, \ref{ass:quadratic tail} hold, then with probability 1 all refining subsequences converge to a Clarke stationary point. We start by recalling the definition of Clarke generalized directional derivative  and Clarke subdifferential.
\begin{definition}[Clarke generalized directional derivative]
Let $f:\mathbb{R}^n\to\mathbb{R}$ be Lipschitz near $x$.  Its \emph{Clarke generalized directional derivative} at $x$ in the direction $h$ is
\[
  f^\circ(x;h)
  :=
  \limsup_{\substack{y\to x\\t\downarrow0}}
    \frac{f(y + t\,h) - f(y)}{t}.
\]
\end{definition}
\begin{definition}[Clarke subdifferential]
Under the same hypotheses, the \emph{Clarke subdifferential} of $f$ at $x$ is
\[
  \partial_C f(x)
  := \bigl\{v\in\mathbb{R}^n : f^\circ(x;h)\;\ge\;\langle v,\,h\rangle
    \quad\forall\,h\in\mathbb{R}^n\bigr\}.
\]
Equivalently, one shows
\[
  \partial_C f(x)
  = \mathrm{conv}\Bigl\{
      \lim_{i\to\infty}\nabla f(x_i):
      x_i\to x,\ f\text{ is differentiable at }x_i
    \Bigr\}.
\]
\end{definition}
\begin{definition}[Clarke-stationary point]\label{def:clarke-stationary}
Let $f:\mathbb{R}^n\to\mathbb{R}$ be locally Lipschitz and let $\partial_{C} f(x)$ denote its
Clarke subdifferential. A point $x^\star\in\mathbb{R}^n$ is \emph{Clarke-stationary} if and only if $0 \in \partial_{C} f(x^\star).$
Equivalently, the Clarke directional derivative satisfies
\[
f^{\circ}(x^\star; d)\ \ge\ 0 \qquad \text{for all } d\in\mathbb{R}^n.
\]
\end{definition}
Now some preliminary results. The next lemma extends Assumption \ref{ass:quadratic tail} to the case where $\alpha$ is a random variable.
\begin{lemma}\label{l:rv}
Let \(A\) be an \(\mathcal F_{k-1}\)-measurable random variable such that
\(A\ge\varepsilon_h\) almost surely. If
Assumption~\ref{ass:quadratic tail} holds, then, for every
\(i\in[0:\Imax]\) and \(m\in[1:\bar m]\), almost surely
\begin{equation*}
\begin{aligned}
    \mathbb P\left(
        |\f_k^{(-1)}-f_k^{(-1)}|
        \ge
        A\Delta_k^{\Dexp}
        \,\middle|\,
        \mathcal F_{k-1}
    \right)
    &\le
    \frac{\varepsilon_h}{A^{\aexp}},
    \\
    \mathbb P\!\left(
|\tilde f_{k,m}^{(i)}-f_{k,m}^{(i)}|
\ge
A\,\Delta_k^{\Dexp}
\ \middle|\ 
\F_{k-1}
\right)
&\le
\frac{\varepsilon_h}{A^{\aexp}}.
\end{aligned}
\end{equation*}
\end{lemma}
\begin{proof}
Similar to \cite[Lemma~2.3]{MR4758401}. It is enough to test the desired conditional inequality against arbitrary
events $G\in\mathcal F_{k-1}$. For simple thresholds
$A=\sum_j a_j\mathbf 1_{G_j}$, with $G_j\in\mathcal F_{k-1}$ and
$a_j\ge\varepsilon_h$, the claim follows by applying
Assumption~\ref{ass:quadratic tail} on each $G\cap G_j$. The general case
$A\ge\varepsilon_h$ follows by taking simple
$\mathcal F_{k-1}$-measurable thresholds $A_n$ such that
$\varepsilon_h\le A_n\le A$ and $A_n\uparrow A$, and then passing to the limit
by dominated convergence.
\end{proof}
To prove $0\in\partial_C f(\bar x)$ at a limit point $\bar x$ of $\{X_k\}$, we proceed in two steps. First, we show that along the (infinite) set $\mathcal{U}$ of \emph{unsuccessful} iterations the forward
difference in the sampled direction is asymptotically nonnegative,
\[
\liminf_{k\in \mathcal{U},\ k\to\infty}\ \frac{f(X_k+\Delta_k D_k)-f(X_k)}{\Delta_k}\ \ge\ 0\quad\text{a.s.}
\]
which is Lemma~\ref{lem:unsuccessful-dir-nonneg} below. 
Second, we exploit that the distribution of $D_k$ has \emph{dense support} on the unit sphere: for any fixed unit vector $u$, there exist indices $k_j\in \mathcal{U}$ with $D_{k_j}\to u$ and $\Delta_{k_j}\to 0$. Passing to the limit and using the outer semicontinuity of the Clarke directional derivative then yields $f^\circ(\bar x;u)\ge 0$ for every $u$, which is equivalent to $0\in\partial_C f(\bar x)$. The next lemma establishes the first step.
\begin{lemma}\label{lem:unsuccessful-dir-nonneg}
Let $\mathcal{U}$ be the (random) set of indices of unsuccessful iterations. Under
Assumptions \ref{ass:compact}, \ref{ass:indep}, \ref{ass:quadratic tail}, and the stepsize/merit
conditions ensuring $\sum_k \EE[\Delta_k^{\Dexp}]<\infty$ (hence $\Delta_k\to 0$ a.s.), we have a.s.
\[
\liminf_{k\in \mathcal{U},\,k\to\infty}\ \frac{f(X_k+\Delta_k D_k)-f(X_k)}{\Delta_k}\ \ge\ 0.
\]
\end{lemma}
\begin{proof}
Fix $l\in\mathbb N$. On an unsuccessful outer iteration $k\in \mathcal{U}$, the inner extrapolation does not
start and for every tested direction $D_{k,m}$, with $m\in[1:\bar{m}]$, the level $i=0$ sufficient-decrease
test fails. Since, on unsuccessful iterations, no direction is accepted, we define the
direction used in the refining subsequence as the last tested direction,
namely $D_k:=D_{k,\bar{m}}$. Hence
$$
\tilde f_k^{(-1)}=\tilde f(X_k),
\qquad
\tilde f_k^{(0)}=\tilde f(X_k+\Delta_kD_k).
$$
Because the iteration is unsuccessful, the sufficient-decrease test fails at
level $i=0$ for this direction. Therefore,
\begin{equation}\label{eq:fail}
\tilde f_k^{(0)}-\tilde f_k^{(-1)}
>
-\theta\Delta_k^{\Dexp}.
\end{equation}
Apply Assumption~\ref{ass:quadratic tail} with the $\F_{k-1}$-measurable threshold
$\Upsilon_k:=\frac{1}{2l}\Delta_k^{1-\Dexp}$ (note $\Upsilon_k\ge \varepsilon_h$ for all $k$ large enough since
$\Delta_k\to 0$ a.s.). Conditionally on $\F_{k-1}$,
\[
\Pb\!\left(\,|\tilde f_k^{(i)}-f_k^{(i)}|\ \ge\ \frac{\Delta_k}{2l}\ \Big|\ \F_{k-1}\right)
\ \le\ \frac{\varepsilon_h}{\Upsilon_k^{\aexp}}
\ =\ \varepsilon_h\,(2l)^{\aexp}\,\Delta_k^{\Dexp},
\  i\in\{-1,0\}.
\]
Taking expectations and summing (Tonelli) over $k$ and $i\in\{-1,0\}$ gives
\[
\sum_{k\ge 0}\ \sum_{i\in\{-1,0\}}\Pb\!\left(|\tilde f_k^{(i)}-f_k^{(i)}|\ \ge\ \frac{\Delta_k}{2l}\right)
\ \le\ \varepsilon_h (2l)^{\aexp}\ \EE\!\left[\sum_{k\ge 0}\Delta_k^{\Dexp}\right]\ <\ \infty.
\]
By Borel--Cantelli (first lemma), there is an a.s.-finite random index $\bar{k}$ such that for all
$k\ge \bar{k}$ (in particular, for all large $k\in \mathcal{U}$),
\begin{equation}\label{eq:err-small}
|\tilde f_k^{(i)}-f_k^{(i)}|\ \le\ \frac{\Delta_k}{2l},\qquad i\in\{-1,0\}.
\end{equation}
For $k\in \mathcal{U}$ and $k\ge \bar{k}$, decompose
\begin{align*}
f(X_k+\Delta_k D_k)-f(X_k)
&=\big(f(X_k+\Delta_k D_k)-\tilde f_k^{(0)}\big)
 +\big(\tilde f_k^{(0)}-\tilde f_k^{(-1)}\big)
 +\big(\tilde f_k^{(-1)}-f(X_k)\big) \\
&\ge -\frac{\Delta_k}{2l}\ +\ \big(\tilde f_k^{(0)}-\tilde f_k^{(-1)}\big)\ -\ \frac{\Delta_k}{2l}
\qquad\text{by \eqref{eq:err-small}} \\
&\ge -\,\theta\,\Delta_k^{\Dexp}\ -\ \frac{\Delta_k}{l}
\qquad\text{by \eqref{eq:fail}}.
\end{align*}

Dividing by $\Delta_k>0$ yields, for all large $k\in \mathcal{U}$,
\[
\frac{f(X_k+\Delta_k D_k)-f(X_k)}{\Delta_k}\ \ge\ -\,\theta\,\Delta_k^{\Dexp-1}\ -\ \frac{1}{l}.
\]
Since $\Delta_k\to 0$ a.s., we obtain a.s.
\[
\liminf_{k\in \mathcal{U},\,k\to\infty}\ \frac{f(X_k+\Delta_k D_k)-f(X_k)}{\Delta_k}
\ \ge\ -\,\frac{1}{l}.
\]
As $l\in\mathbb N$ was arbitrary, letting $l\to\infty$ gives the claim.
\end{proof}
We now report the main convergence result for our stochastic direct-search scheme. The result requires the existence of accumulation points for the sequence $\left\{x_{k}\right\}$, which can be obtained assuming that the iterates generated by the algorithm lie in a compact set (see Assumption \ref{ass:compact}). We omit the proof, as it follows verbatim from \cite[Theorem 3.3]{MR4758401}.
\begin{theorem}
Assume that \(f\) is Lipschitz continuous in a
neighborhood of every limit point of the sequence of iterates
\(\{X_k\}\). Let \(\mathcal U\) denote the random set of indices of
unsuccessful iterations. Suppose that
Assumptions~\ref{ass:compact}, \ref{ass:indep},
\ref{ass:quadratic tail}, and conditions~\eqref{ass:eta} and
\eqref{ass:theta_eta_compatibility} hold.
Then, with probability one, the following implication holds: if
\(V\subseteq\mathcal U\) is a random index set such that
\(\{D_k\}_{k\in V}\) is dense in \(\mathbb S^{n-1}\) and
\(
    \lim_{\substack{k\to\infty\\ k\in V}} X_k = X^*,
\)
then \(X^*\) is Clarke stationary, namely,
\(
    f^\circ(X^*;d)\ge0
\) for every $d\in\mathbb R^n$.
\end{theorem} 

%% file: sections/complexity_goldstein.tex
\section{Expected Complexity with Goldstein-stationarity stopping}\label{sec:clarke complexity}

To establish the expected iteration complexity of the proposed algorithm, we adapt the supermartingale stopping-time framework introduced in \cite{MR4151319}. This framework allows us to bound the expected number of iterations required to reach an approximate stationary condition. In contrast to the almost-sure convergence analysis of the previous section, we now characterize stationarity via the Goldstein $\goldelta$-subdifferential, $\partial_\goldelta f$, rather than the Clarke subdifferential, $\partial_C f$. The reason for this shift is detailed in Remark~\ref{rem:rescale-Lipschitz}.
We start by recalling the definition of stopping times.
\begin{definition}[Stopping time]
Let $(\mathcal A_k)_{k\ge0}$ be a filtration on
$(\Omega,\mathcal A,\mathbb P)$. A random variable
\(
    T:\Omega\to\mathbb N_0\cup\{\infty\}
\)
is called a stopping time with respect to $(\mathcal A_k)_{k\ge0}$ if
\[
    \{T\le k\}\in\mathcal A_k
    \qquad
    \forall k\in\mathbb N_0 .
\]
Equivalently, \(T\) is a stopping time if
\[
    \{T=k\}\in\mathcal A_k
    \qquad
    \forall k\in\mathbb N_0 .
\]
\end{definition}
For the complexity analysis, we use the pre-direction algorithmic-history filtration
\[
    \mathcal A_k:=\mathcal G_{k-1},
\]
where \(\mathcal G_{k-1}\) is the information available at the beginning of
iteration \(k\), before sampling the directions and before drawing the oracle
estimates of iteration \(k\). Thus \(X_k\), \(\Delta_k\), and \(\Phi_k\) are
\(\mathcal A_k\)-measurable. Moreover, with the notation of Section \ref{sec:algo},
\[
    \mathcal A_k=\mathcal G_{k-1}
    \subseteq
    \mathcal F_{k-1}
    =
    \mathcal G_{k-1}
    \vee
    \sigma(D_{k,1},\ldots,D_{k,\bar{m}})
    \subseteq
    \mathcal A_{k+1}=\mathcal G_k .
\]
The filtration \((\mathcal F_{k-1})\) is used for oracle-conditioning
arguments after the directions have been sampled, whereas
\((\mathcal A_k)\) is used for the stopping-time and radius-dynamics
arguments. Since \(f\) is possibly nonsmooth, the stopping time is defined in terms of the
\(\goldelta\)-Goldstein subdifferential.
\begin{definition}[Goldstein subdifferential]
For \(r>0\), let
$B_r(x):=\{y\in\mathbb{R}^n:\|y-x\|\le r\}
$
denote the closed ball of radius \(r\) centered at \(x\). The
\emph{\(r\)-Goldstein subdifferential} of \(f\) at \(x\) is defined as the
convex hull of the Clarke subdifferentials at all points \(y\in B_r(x)\), namely
\[
\partial_r f(x)
:=
\operatorname{conv}\left(
\bigcup_{y\in B_r(x)} \partial_C f(y)
\right).
\]
\end{definition}
For fixed tolerances $\goldelta>0$ and $\varepsilon>0$, we define the stopping time as
\begin{equation}\label{Teps}
    T_{\goldelta, \varepsilon} = \inf\left\{ k \in \mathbb{N}_0 : \min_{g \in \partial_\goldelta f(\Xk)} \|g\| \leq \varepsilon \right\}.
\end{equation}
Equivalently, \(T_{\goldelta,\varepsilon}\) is the first iteration at which
\(X_k\) is a \((\goldelta,\varepsilon)\)-Goldstein stationary point.\\
Since the definition only depends on \((\Xk)\) which are $\mathcal A_{k}$ measurable, it follows that $T_{\goldelta, \varepsilon}$ is indeed a ($\mathcal A_k$)-stopping time. Adapting the framework of \cite{MR4151319}, we can provide a bound on $\me[T_{\goldelta, \varepsilon}]$. At a given iteration \(k\), define the base-estimate event
\[
\mathcal E_{k,0}
:=
\left\{
    |\tilde f(X_k)-f(X_k)|
    \le
    \varepsilon_f\Delta_k^{\Dexp}
\right\},
\]
and, for each \(m=1,\ldots,\bar{m}\), the trial-estimate event
\[
\mathcal E_{k,m}
:=
\left\{
    |\tilde f(X_k+\Delta_kD_{k,m})
      -f(X_k+\Delta_kD_{k,m})|
    \le
    \varepsilon_f\Delta_k^{\Dexp}
\right\}.
\]
Then, fix a stationarity tolerance
\(\varepsilon>0\) and a Goldstein neighborhood radius \(\goldelta>0\).
Let
\[
    G_{\goldelta,k}^*
    \in
    \operatorname*{argmin}_{g\in\partial_{\goldelta}f(X_k)}\|g\|,
\]
and, on the event \(\{T_{\goldelta,\varepsilon}>k\}\), define
\[
    D_k^*
    :=
    -\frac{G_{\goldelta,k}^*}{\|G_{\goldelta,k}^*\|}.
\]
Indeed, on this event we have \(\|G_{\goldelta,k}^*\|>\varepsilon\), so
\(D_k^*\) is well-defined. The direction \(D_k^*\) is the Goldstein descent direction at iteration $k$. Since the algorithm
samples random directions, we only need one sampled direction to be sufficiently
aligned with \(D_k^*\). To quantify the probability that a randomly sampled direction is sufficiently
aligned with a target descent direction, we use the standard formula for the
surface measure of a spherical cap. 

\begin{definition}[Spherical cap]
\label{def:spherical-cap-probability}
Fix \(0<\rho<1\) and a unit vector \(u\in\mathbb S^{n-1}\). The
\emph{spherical cap with center \(u\) and alignment parameter \(\rho\)} is
\[
    \mathcal C_\rho(u)
    :=
    \{d\in\mathbb S^{n-1}:\langle u,d\rangle\ge\rho\}.
\]
Equivalently, \(\mathcal C_\rho(u)\) is the cap with half-angle
\(\arccos(\rho)\). 
\end{definition}
Thus, \(D_{k,m}\in\mathcal C_\rho(D_k^*)\) means that the sampled direction
\(D_{k,m}\) is sufficiently aligned with the Goldstein descent direction at iteration $k$. Now, define the good event that both estimates are accurate and  $\exists\ m\in\{1,\ldots,\bar{m}\}:
D_{k,m}\in \mathcal C_\rho(D_k^\ast)$, with \(D_k^*\) being the Goldstein descent direction at iteration $k$, and \(C_\rho(D_k^\ast)\) being the spherical cap therein centered. That is, define
\[
\mathcal I_k
:=
\mathcal E_{k,0}
\cap
\left[
\bigcup_{m=1}^{\bar{m}}
\left(
\{D_{k,m}\in\mathcal C_\rho(D_k^*)\}
\cap
\mathcal E_{k,m}
\right)
\right],
\]
and let $J_k:=\mathbf 1_{\mathcal I_k}. $
We then set
\begin{equation} \label{eq:Wk}
    W_{k+1}:=2J_k-1\in\{-1,+1\}.
\end{equation}
Hence \(W_{k+1}=+1\) if the good event \(\mathcal I_k\) occurs, and
\(W_{k+1}=-1\) otherwise. In particular,
we will show that,
whenever \(T_{\goldelta,\varepsilon}>k\) and
\(\Delta_k\le \bar\Delta_{\goldelta,\varepsilon}\), the occurrence of \(\mathcal I_k\)
implies that iteration \(k\) is successful. We set \(W_0=1\) only for
initialization. Notice that \(W_k\) is \(\mathcal A_k\)-measurable, whereas
\(W_{k+1}\) is not \(\mathcal A_k\)-measurable. We now state the abstract conditions needed to apply the stopping-time
framework (see also \cite[Assumption 1]{MR4151319}).

\begin{assumption}
\label{ass:5}
Let \(T_{\goldelta,\varepsilon}\) be the stopping time defined in
\eqref{Teps}. Let the stochastic process \(\{(\Phi_k,\Delta_k,W_k)\}_{k\ge0}\) be adapted to
\((\mathcal A_k)_{k\ge0}\). The following conditions hold.

\begin{enumerate}
    \item[(i)] There exist constants \(\lambda>0\) and
    \(\Delta_{\max}= \Delta_0 e^{\lambda j_{max}}>0\), \(j_{max}\in\mathbb Z\) such that
    \[
        \lambda=-\log\gamma
        \qquad\text{and}\qquad
        \Delta_k\le \Delta_{\max}
        \quad\text{for all }k\ge0.
    \]

    \item[(ii)] There exist a threshold
    \[
        \bar\Delta_{\goldelta,\varepsilon}
        =
        \Delta_0 e^{\lambda j_{\goldelta,\varepsilon}},
        \qquad
        j_{\goldelta,\varepsilon}\in\mathbb Z,
        \quad j_{\goldelta,\varepsilon}\le 0,
    \]
    and a constant \(q>1/2\) such that
    \[
        \mathbb P(W_{k+1}=+1\mid \mathcal A_k)= q,
        \qquad
        \mathbb P(W_{k+1}=-1\mid \mathcal A_k)= 1-q,
    \]
    and, for every \(k\ge0\),
    \begin{equation}
    \label{eq:dynamics}
        \mathbf 1_{\{T_{\goldelta,\varepsilon}>k\}}
        \Delta_{k+1}
        \ge
        \mathbf 1_{\{T_{\goldelta,\varepsilon}>k\}}
        \min\left\{
            \Delta_k e^{\lambda W_{k+1}},
            \bar\Delta_{\goldelta,\varepsilon}
        \right\}.
    \end{equation}

    \item[(iii)] There exist a nondecreasing function
    \(h:[0,\infty)\to [0,\infty)\) and a constant \(\Theta>0\) such that,
    for every \(k\ge0\),
    \begin{equation}
    \label{eq:abstract_drift}
        \mathbb E[
            \Phi_k-\Phi_{k+1}
            \mid
            \mathcal A_k
        ]
        \mathbf 1_{\{T_{\goldelta,\varepsilon}>k\}}
        \ge
        \Theta h(\Delta_k)
        \mathbf 1_{\{T_{\goldelta,\varepsilon}>k\}}
        \qquad\text{a.s.}
    \end{equation}
    In particular, we take
    $h(\Delta)=\Delta^{\Dexp}. $
\end{enumerate}
\end{assumption}
Assumption~\ref{ass:5} states that, before the stopping time, the
nonnegative stochastic process \(\Phi_k\) has an expected decrease of at least
\(\Theta h(\Delta_k)\) at each iteration. Moreover, before the stopping time
and whenever \(\Delta_k\le \bar\Delta_{\goldelta,\varepsilon}\), the radius
has an upward tendency: the good event \(\mathcal I_k\), which implies
\(W_{k+1}=+1\), occurs with conditional probability larger that  \(1/2\), namely
\[\mathbb P(W_{k+1}=+1\mid\mathcal A_k)
   = q >1/2.
\]
Our goal is to bound \(\mathbb E[T_{\goldelta,\varepsilon}]\) in terms of
\(h(\bar\Delta_{\goldelta,\varepsilon})\). What happens is that, on average, $\D_k \geq \bar \D_{\goldelta,\varepsilon}$ frequently, and hence, $\me[\Phi_{k+1} - \Phi_k]$ can be bounded by a negative fixed value (dependent on $\goldelta, \varepsilon$), sufficiently frequently, which will allow us
to apply Wald’s identity and derive the bound on $\me[T_{\goldelta,\varepsilon}]$ (for more details refer to \cite[Theorem 2]{MR4151319}). 
The next result is a direct specialization of the supermartingale stopping-time bound of Blanchet, Cartis, Menickelly and Scheinberg~\cite[Theorem~2]{MR4151319} to our framework.
\begin{theorem}\label{t:expected complexity}
    Let Assumption~\ref{ass:5} hold. Then
    \[
        \mathbb E[T_{\goldelta,\varepsilon}]
        \le
        \frac{q}{2q-1}
        \cdot
        \frac{\Phi_0}{\Theta h(\bar\Delta_{\goldelta,\varepsilon})}
        +1 .
    \]
\end{theorem}
Assumption~\ref{ass:5}(i) follows from the compactness assumption (see Remark \ref{rem:compact-tested-points}). For Assumption~\ref{ass:5}(iii), the merit-function drift established in Lemma~\ref{lem:phi_drift_main} gives
\[
    \mathbb E[
        \Phi_k-\Phi_{k+1}
        \mid
        \mathcal F_{k-1}
    ]
    \ge
    \Theta\Delta_k^{\Dexp},
\]
where
    $\Theta:=\eta(1-\gamma^{\Dexp})-c_{\max}>0.$
Since
\(
    \mathcal A_k\subseteq\mathcal F_{k-1},
\)
the tower property gives
\[
\begin{aligned}
    \mathbb E[
        \Phi_k-\Phi_{k+1}
        \mid
        \mathcal A_k
    ]
    =
    \mathbb E\!\left[
        \mathbb E[
            \Phi_k-\Phi_{k+1}
            \mid
            \mathcal F_{k-1}
        ]
        \,\middle|\,
        \mathcal A_k
    \right]                                                   
    \ge
    \Theta\Delta_k^{\Dexp}.
\end{aligned}
\]
Multiplying by the \(\mathcal A_k\)-measurable indicator
\(\mathbf 1_{\{T_{\goldelta,\varepsilon}>k\}}\), we obtain
\[
    \mathbb E[
        \Phi_k-\Phi_{k+1}
        \mid
        \mathcal A_k
    ]
    \mathbf 1_{\{T_{\goldelta,\varepsilon}>k\}}
    \ge
    \Theta\Delta_k^{\Dexp}
    \mathbf 1_{\{T_{\goldelta,\varepsilon}>k\}},
\]
which is Assumption~\ref{ass:5}(iii).\\
It remains to verify Assumption~\ref{ass:5}(ii), namely the radius dynamics and
the lower bound
\[
    \mathbb P(W_{k+1}=+1\mid\mathcal A_k)\ge q>1/2
\]
before the stopping time and for sufficiently small \(\Delta_k\). The next two
lemmas provide exactly this. The first lemma proves a uniform lower bound
on the probability that at least one of the \(\bar{m}\) sampled directions yields a
sufficient-decrease step. The second lemma shows that, with the given stepsize update rule, the radius process satisfies the required comparison
inequality. Before stating the lemmas, we report the following definition that collects the
notation for the regularized incomplete Beta function and the cap
probability under the uniform distribution of directions on the sphere.

\begin{definition}[Regularized incomplete Beta function and Spherical-cap probability]
\label{def:reg-inc-beta-funct}
 For \(a,b>0\) and \(z\in[0,1]\), the
\emph{regularized incomplete Beta function} is
\[
    I_z(a,b)
    :=
    \frac{B(z;a,b)}{B(a,b)},
    \qquad
    B(z;a,b)
    :=
    \int_0^z t^{a-1}(1-t)^{b-1}\,dt,
\]
where
\[
    B(a,b)
    :=
    \frac{\Gamma(a)\Gamma(b)}{\Gamma(a+b)}.
\]
Equivalently, if \(U\sim\mathrm{Beta}(a,b)\), then
\[
    I_z(a,b)=\mathbb P(U\le z).
\]
Now, assume \(n\ge2\). By standard properties of a uniformly distributed point on the sphere (see, e.g., \cite[Chapter~9]{MardiaJupp2000}),
the surface measure of the spherical cap \(\mathcal C_\rho(u)\) is independent
of its center \(u\). More precisely, if \(D\sim\mathrm{Unif}(\mathbb S^{n-1})\), then
\begin{equation}\label{capprob}
    \mathbb P(D\in\mathcal C_\rho(u))
    =
    p_\rho
    :=
    \frac12
    I_{1-\rho^2}
    \left(
        \frac{n-1}{2},
        \frac12
    \right).
\end{equation}
 Moreover,
$p_\rho>0
    \ \text{for every fixed }\rho\in(0,1),
    \ \text{and}\ 
    p_\rho\downarrow0
     \text{ as }\rho\uparrow1.$

\end{definition}
If
\(D_{k,m}\) is uniformly distributed on \(\mathbb S^{n-1}\), then $\mathbb P(D_{k,m}\in \mathcal C_\rho(D_k^*)\mid\mathcal A_k)
    =
    p_\rho.$
We first show that, before the stopping time and for sufficiently small
stepsizes, the occurrence of the good event \(\mathcal I_k\) forces the algorithmic
sufficient-decrease test to succeed. This is the analogue, in the present
Goldstein-stationarity setting, of the key implication used in Dzahini's
analysis \cite{MR4359469}: if the stationarity measure is still larger than the prescribed
tolerance, the stepsize is small enough, and the estimates are accurate, then
the iteration must be successful. Here, the role of the descent direction is
played by any sampled direction lying in the spherical cap~
\(\mathcal C_\rho(D_k^\ast)\).

\begin{lemma}[Sufficient decrease on the good event]
\label{lem:good-event-sufficient-decrease}
Assume
\[
   c_1
   :=
   \rho\varepsilon-\sqrt{1-\rho^{2}}\,L
   >
   0 .
\]
Let \(\theta>0\) be the sufficient-decrease parameter used by the algorithm, and
choose
\[
  \Delta_{\goldelta,\varepsilon}
  <
  \min\left\{
      \goldelta,
      \left(
          \frac{c_1}{\theta+2\varepsilon_f}
      \right)^{1/(\Dexp-1)}
  \right\}.
\]
Let
\[
    B_k
    :=
    \{T_{\goldelta,\varepsilon}>k\}
    \cap
    \{\Delta_k\le\Delta_{\goldelta,\varepsilon}\}.
\]
Then, on the event \(B_k\cap\mathcal I_k\), there exists
\(m\in\{1,\ldots,\bar m\}\) such that
\[
    D_{k,m}\in\mathcal C_\rho(D_k^*)
\]
and both the base estimate and the corresponding trial estimate are
\(\varepsilon_f\)-accurate. For such an index \(m\),
\[
    \tilde f(X_k+\Delta_kD_{k,m})
    -
    \tilde f(X_k)
    \le
    -\theta\Delta_k^{\Dexp}.
\]
Consequently, iteration \(k\) is successful.
\end{lemma}

\begin{proof}
We argue on the event \(B_k\). By definition of
\(T_{\goldelta,\varepsilon}\),
\[
    \min_{g\in\partial_{\goldelta} f(X_k)}\|g\|>\varepsilon .
\]
Choose
\[
    G_{\goldelta,k}^*
    \in
    \operatorname*{argmin}_{g\in\partial_{\goldelta} f(X_k)}
    \|g\|,
    \qquad
    D_k^*
    :=
    -\frac{G_{\goldelta,k}^*}{\|G_{\goldelta,k}^*\|}.
\]
Then \(\|G_{\goldelta,k}^*\|>\varepsilon\).

Fix \(m\in\{1,\ldots,\bar{m}\}\) such that
\(D_{k,m}\in\mathcal C_\rho(D_k^*)\), and set $ S_{k,m}:=\Delta_kD_{k,m}.$
By Lebourg's mean-value theorem, there exist
\[
    \bar X_{k,m}=X_k+\Xi_{k,m}S_{k,m},
    \qquad
    \Xi_{k,m}\in(0,1),
\]
and
\[
    \bar G_{k,m}\in\partial_C f(\bar X_{k,m})
\]
such that
\[
    f(X_k+S_{k,m})-f(X_k)
    =
    \langle \bar G_{k,m},S_{k,m}\rangle
    =
    \Delta_k\langle \bar G_{k,m},D_{k,m}\rangle .
\]
Since \(\Delta_k\le\Delta_{\goldelta,\varepsilon}\le\goldelta\), we have
\[
    \|\bar X_{k,m}-X_k\|\le\Delta_k\le\goldelta .
\]
Therefore \(\bar X_{k,m}\in B_{\goldelta}(X_k)\), and hence
\[
    \bar G_{k,m}\in\partial_{\goldelta}f(X_k).
\]
By the projection theorem applied to the closed convex set
\(\partial_{\goldelta} f(X_k)\),
\[
    \langle \bar G_{k,m},G_{\goldelta,k}^*\rangle
    \ge
    \|G_{\goldelta,k}^*\|^2 .
\]
Consequently,
\[
    \langle \bar G_{k,m},D_k^*\rangle
    =
    -\frac{
        \langle \bar G_{k,m},G_{\goldelta,k}^*\rangle
    }{
        \|G_{\goldelta,k}^*\|
    }
    \le
    -\|G_{\goldelta,k}^*\|
    <
    -\varepsilon .
\]
Write
\[
    D_{k,m}
    =
    R_{k,m}D_k^*
    +
    \sqrt{1-R_{k,m}^2}\,V_{k,m},
\]
where
\[
    R_{k,m}:=\langle D_k^*,D_{k,m}\rangle,
    \qquad
    V_{k,m}\perp D_k^*,
    \qquad
    \|V_{k,m}\|=1.
\]
Since \(D_{k,m}\in\mathcal C_\rho(D_k^*)\), we have
\(R_{k,m}\ge\rho\). Moreover, since \(f\) is \(L\)-Lipschitz on the relevant
neighborhood, \(\|\bar G_{k,m}\|\le L\). Thus
\[
\begin{aligned}
    \langle \bar G_{k,m},D_{k,m}\rangle
    &=
    R_{k,m}\langle \bar G_{k,m},D_k^*\rangle
    +
    \sqrt{1-R_{k,m}^2}
    \langle \bar G_{k,m},V_{k,m}\rangle                                      \\
    &\le
    -R_{k,m}\varepsilon
    +
    \sqrt{1-R_{k,m}^2}\,L                                                   \\
    &\le
    -\rho\varepsilon+\sqrt{1-\rho^2}\,L                                      \\
    &=
    -c_1 .
\end{aligned}
\]
Hence
\[
    f(X_k+\Delta_kD_{k,m})-f(X_k)
    \le
    -c_1\Delta_k.
\]
If the two estimates are \(\varepsilon_f\)-accurate, then
\[
\begin{aligned}
    \tilde f(X_k+\Delta_kD_{k,m})-\tilde f(X_k)
    &\le
    f(X_k+\Delta_kD_{k,m})-f(X_k)
    +
    2\varepsilon_f\Delta_k^{\Dexp}                                           \\
    &\le
    -c_1\Delta_k
    +
    2\varepsilon_f\Delta_k^{\Dexp}.
\end{aligned}
\]
By the definition of \(\Delta_{\goldelta,\varepsilon}\),
\[
    \Delta_k^{\Dexp-1}
    \le
    \Delta_{\goldelta,\varepsilon}^{\Dexp-1}
    <
    \frac{c_1}{\theta+2\varepsilon_f}.
\]
Therefore
\[
    c_1\Delta_k
    >
    (\theta+2\varepsilon_f)\Delta_k^{\Dexp}.
\]
It follows that
\[
    \tilde f(X_k+\Delta_kD_{k,m})-\tilde f(X_k)
    \le
    -\theta\Delta_k^{\Dexp}.
\]
Thus the sufficient-decrease test succeeds for direction \(m\), and iteration
\(k\) is successful.
\end{proof}

\begin{remark}[Overcoming the coarse Clarke bound via Goldstein subdifferentials]%
\label{rem:rescale-Lipschitz}
In a pure Clarke framework, bounding the Lebourg subgradient relies on the coarse Lipschitz bound $\|\bar G_{k,m}-G_k\|\le 2L$. This requires a descent margin of $\rho\varepsilon-2L>0$, forcing $\rho > 2L/\varepsilon$. As $\varepsilon \to 0$, this constraint inevitably demands $\rho \ge 1$. The spherical cap becomes empty, collapsing the success probability $p_\rho$ to zero, and restricting the analysis to low-accuracy targets where $\varepsilon > 2L$. Shifting to the Goldstein $\goldelta$-subdifferential resolves this topological barrier. By targeting an $(\varepsilon, \goldelta)$-Goldstein stationary point and restricting $\Delta_k \le \goldelta$, the intermediate Lebourg point is trapped in $B_\goldelta(X_k)$, ensuring $\bar G_{k,m} \in \partial_\goldelta f(X_k)$. The projection theorem then eliminates the $2L$ residual, refining the descent condition to $c_1 := \rho\varepsilon - \sqrt{1-\rho^2}L > 0$. This yields $\rho > \frac{L}{\sqrt{\varepsilon^2 + L^2}}$, which is strictly less than $1$ for all $\varepsilon > 0$, guaranteeing a positive success probability as $\varepsilon \to 0$. In the smoother regime \(f\in C^{1,1}\), define 
\(
    G_k:=\nabla f(X_k),
    \;
    \bar G_{k,m}:=\nabla f(\bar X_{k,m}),
\)
where \(\bar X_{k,m}\) is the intermediate point on the segment
\([X_k,X_k+\Delta_kD_{k,m}]\). If \(\nabla f\) is
\(\bar L\)-Lipschitz, then
\[
    \left|
    \langle \bar G_{k,m}-G_k,D_{k,m}\rangle
    \right|
    \le
    \bar L\Delta_k.
\]
The $2L$ penalty is replaced by $\mathcal{O}(\Delta_k)$, yielding $c_{1,k}=\rho\varepsilon-\mathcal{O}(\Delta_k)$. Since $\Delta_k\to 0$ almost surely, $c_{1,k}>0$ eventually holds regardless of $\varepsilon$, as established in~\cite{MR4359469}.
\end{remark}
We next quantify the probability of the good event introduced above. More specifically, the following lemma shows that,
under conditional independence and accuracy assumptions, this event occurs
with probability bounded from below uniformly in \(k\).

\begin{lemma}[Probability of the good event]
\label{lem:good-event-probability}
Assume that, conditionally on \(\mathcal A_k\), the directions
\(D_{k,1},\ldots,D_{k,\bar{m}}\) are independent and uniformly distributed on
\(\mathbb S^{n-1}\). Assume also that the estimates are
\(\beta\)-probabilistically \(\varepsilon_f\)-accurate, namely
\[
    \mathbb P(\mathcal E_{k,0}\mid\mathcal A_k)\ge\beta,
    \qquad
    \mathbb P(\mathcal E_{k,m}\mid\mathcal A_k,D_{k,m})\ge\beta,
    \quad m=1,\ldots,\bar{m}.
\]
Moreover, assume that, conditionally on \(\mathcal A_k\), the base-estimate
event \(\mathcal E_{k,0}\) is independent of the trial events, and that the
pairs
\[
    \{(D_{k,m},\mathcal E_{k,m})\}_{m=1}^{\bar{m}}
\]
are conditionally independent.\\
Then,
we have
\[
    \mathbb P(\mathcal I_k\mid\mathcal A_k)
    =
    \mathbb P(W_{k+1}=+1\mid\mathcal A_k)
    \ge
    q_{\bar{m}}\,
\]
where
\[
    q_{\bar{m}}
    :=
    \beta\left[1-(1-p_\rho\beta)^{\bar{m}}\right].
\]
\end{lemma}

\begin{proof}
Since \(D_{k,m}\) is uniformly distributed on
\(\mathbb S^{n-1}\) conditionally on \(\mathcal A_k\), and since the spherical
cap has probability \(p_\rho\), we have
\[
    \mathbb P(D_{k,m}\in\mathcal C_\rho(D_k^*)\mid\mathcal A_k)=p_\rho .
\]
Define  $\mathcal S_{k,m}
    :=
    \{D_{k,m}\in\mathcal C_\rho(D_k^*)\}
    \cap
    \mathcal E_{k,m}. $
Using the conditional uniformity of \(D_{k,m}\) and the conditional
\(\beta\)-accuracy of the trial estimate,
\[
\begin{aligned}
    \mathbb P(\mathcal S_{k,m}\mid\mathcal A_k)
    &=
    \mathbb E\!\left[
        \mathbf 1_{\{D_{k,m}\in\mathcal C_\rho(D_k^*)\}}
        \mathbb P(\mathcal E_{k,m}\mid\mathcal A_k,D_{k,m})
        \,\middle|\,
        \mathcal A_k
    \right]                                      \\
    &\ge
    \beta\,
    \mathbb P(D_{k,m}\in\mathcal C_\rho(D_k^*)\mid\mathcal A_k)              \\
    &=
    p_\rho\beta .
\end{aligned}
\]
By conditional independence, the events
\(\mathcal S_{k,1},\ldots,\mathcal S_{k,\bar{m}}\) are conditionally independent
given \(\mathcal A_k\). Hence
\[
    \mathbb P\left(
        \bigcup_{m=1}^{\bar{m}} \mathcal S_{k,m}
        \,\middle|\,
        \mathcal A_k
    \right)
    \ge
    1-(1-p_\rho\beta)^{\bar{m}}.
\]
Furthermore, \(\mathcal E_{k,0}\) is conditionally independent of these trial
events and satisfies
\[
    \mathbb P(\mathcal E_{k,0}\mid\mathcal A_k)\ge\beta.
\]
Therefore,
\[
\begin{aligned}
    \mathbb P(\mathcal I_k\mid\mathcal A_k)
    &=
    \mathbb P\left(
        \mathcal E_{k,0}
        \cap
        \bigcup_{m=1}^{\bar{m}} \mathcal S_{k,m}
        \,\middle|\,
        \mathcal A_k
    \right)                                                               \\
    &\ge
    \beta\left[1-(1-p_\rho\beta)^{\bar{m}}\right]                                  \\
    &=
    q_{\bar{m}}.
\end{aligned}
\]
Since \(W_{k+1}=+1\) if and only if \(\mathcal I_k\) occurs, we obtain
\[
    \mathbb P(W_{k+1}=+1\mid\mathcal A_k)
    =
    \mathbb P(\mathcal I_k\mid\mathcal A_k)
    \ge q_{\bar{m}},
\]
which gives the claim.
\end{proof}

\begin{remark}[Polling capacity \(\bar{m}\) and dimensional complexity]
\label{rem:M-lower-bound}
To apply the expected complexity bound of Theorem~\ref{t:expected complexity}
with constants bounded uniformly, it is convenient to fix a number
\(q_0\in(1/2,\beta)\) and choose \(\bar{m}\) so that
\[
    q_{\bar{m}}
    :=
    \beta\,[\,1-(1-p_\rho\beta)^{\bar{m}}\,]
    \ge q_0 .
\]
This is possible only if \(\beta>1/2\). Solving the inequality
\(q_{\bar{m}}\ge q_0\) gives
\[
    (1-p_\rho\beta)^{\bar{m}}
    \le
    1-\frac{q_0}{\beta},
\]
and hence
\[
    \bar{m}
    \ge
    \frac{\log(1-q_0/\beta)}
         {\log(1-p_\rho\beta)}.
\]
Thus, for example, it is sufficient to choose
\[
    \bar{m}
    >
    \bar{m}_{\min}(q_0)
    :=
    \left\lceil
    \frac{\log(1-q_0/\beta)}
         {\log(1-p_\rho\beta)}
    \right\rceil .
\]
With this choice, \(q_{\bar{m}}\ge q_0>1/2\), satisfying the supermartingale condition. However, targeting the Goldstein $\goldelta$-subdifferential requires the descent margin $c_1 > 0$, forcing $\rho > L / \sqrt{\varepsilon^2 + L^2}$.
As the stationarity tolerance \(\varepsilon\downarrow0\), the required cap
alignment satisfies \(\rho\uparrow1\), and the cap probability \(p_\rho\)
goes to zero. To quantify the resulting dependence of \(\bar{m}\) on \(\varepsilon\), choose
\(\rho\) so that
\[
    1-\rho^2\asymp \left(\frac{\varepsilon}{L}\right)^2 .
\]
Equivalently, the cap remains nonempty but becomes narrower as
\(\varepsilon\downarrow0\). Let \(z:=1-\rho^2\). Using the small-\(z\)
asymptotics of the regularized incomplete Beta function from Definition~\ref{def:reg-inc-beta-funct}, we have
\[
    p_\rho
    =
    \frac12
    I_z\left(\frac{n-1}{2},\frac12\right)
    \asymp
    z^{(n-1)/2}.
\]
Therefore
\[
    p_\rho
    \asymp
    \left(\frac{\varepsilon}{L}\right)^{n-1}.
\]
Since $\log(1-p_\rho\beta)\asymp -p_\rho\beta
    $ as $p_\rho\downarrow0, $
the sufficient polling budget satisfies
\[
    \bar{m}_{\min}(q_0)
    =
    \mathcal O\left(\frac{1}{p_\rho}\right)
    =
    \mathcal O\left(
        \left(\frac{L}{\varepsilon}\right)^{n-1}
    \right).
\]
Thus, for fixed \(L\), choosing
\[
    \bar{m}=\mathcal O(\varepsilon^{1-n})
\]
is sufficient to keep the success probability uniformly bounded away from
\(1/2\). This dimensional dependence reflects the geometric difficulty of
finding a descent direction in a purely nonsmooth landscape via a dense set of random directions.
The above argument, together with the use of the spherical cap, is somehow related to the probabilistic-descent framework of
Gratton et al.~\cite[Appendix~B]{GrattonRoyerVicenteZhang2015}. There, for
smooth objectives, random polling directions are used to obtain, with high
probability, a direction lying in a spherical cap around the negative gradient;
and a fixed positive cosine threshold  (for
instance of order \(1/\sqrt n\)) is sufficient to guarantee descent. In our nonsmooth
Goldstein setting, the spherical-cap idea is used, but nonsmoothness lets the cap shrink with the stationarity
tolerance, leading to the polling budget
\(\bar m=O(\varepsilon^{1-n})\). We finally note that the constant
\(c_{\max}=2\mu_h \bar{m}(\Imax+1)\) depends on \(\bar{m}\). We can however
guarantee that \(c_{\max}\) remains uniformly bounded by tightening of
the sample-average accuracy, that is by increasing the batch size in the stochastic estimates construction, if needed.
\end{remark}
Recall that an iteration is called \emph{successful} if at least one of
the \(\bar{m}\) tested directions satisfies the sufficient-decrease test, and
\emph{unsuccessful} otherwise. Equivalently, success corresponds to
\(H_k\ge 0\), whereas failure corresponds to \(H_k=-1\). By the stepsize
update rule, an unsuccessful iteration contracts the polling radius by
\(\gamma\), while a successful iteration expands the next polling radius by at
least \(\gamma^{-1}\). The good-event indicator used in the stopping-time argument is instead
\(J_k=\mathbf 1_{\mathcal I_k}\). Hence \(J_k=1\) does not define success in
general; rather, on the event
\(\{T_{r,\varepsilon}>k\}\cap\{\Delta_k\le\Delta_{r,\varepsilon}\}\),
Lemma~\ref{lem:good-event-sufficient-decrease} shows that
\(\mathcal I_k\) implies \(H_k\ge0\). The following lemma verifies the radius-dynamics condition in
Assumption~\ref{ass:5}.

\begin{lemma}[Verification of Assumption~\ref{ass:5}(ii)]
\label{lem:assumption-ii-holds}
Let \(\gamma\in(0,1)\) and set
$ \lambda:=-\log\gamma>0.$
Let
\[
    \bar\Delta_{\goldelta,\varepsilon}
    =
    \Delta_0\gamma^{j_{\goldelta,\varepsilon}},
    \qquad
    j_{\goldelta,\varepsilon}\in\mathbb Z,\qquad j_{r,\varepsilon} \leq 0,
\]
be such that
\[
    \bar\Delta_{\goldelta,\varepsilon}
    \le
    \Delta_{\goldelta,\varepsilon},
\]
where \(\Delta_{\goldelta,\varepsilon}\) is defined in
Lemma~\ref{lem:good-event-sufficient-decrease}. Then Assumption~\ref{ass:5}(ii) holds with
\[
    q=q_{\bar{m}}
    :=
    \beta\left[1-(1-p_\rho\beta)^{\bar{m}}\right].
\]
\end{lemma}

\begin{proof}
The inequality in Assumption~\ref{ass:5}(ii) is trivial when
\(\mathbf 1_{\{T_{\goldelta,\varepsilon}>k\}}=0\). Hence, assume that
\(T_{\goldelta,\varepsilon}>k\). Since the stepsizes lie on the geometric mesh, we can write
\[
    \Delta_k=\Delta_0\gamma^{i_k}
    \qquad\text{for some } i_k\in\mathbb Z.
\]
If \(\Delta_k>\bar\Delta_{\goldelta,\varepsilon}\), then
\[
    \Delta_k\ge \gamma^{-1}\bar\Delta_{\goldelta,\varepsilon}.
\]
Moreover, from the stepsize update rule, we always have
\[
    \Delta_{k+1}\ge \gamma\Delta_k.
\]
Thus
\[
    \Delta_{k+1}
    \ge
    \gamma\Delta_k
    \ge
    \bar\Delta_{\goldelta,\varepsilon},
\]
and therefore
\[
    \Delta_{k+1}
    \ge
    \min\left\{
        \Delta_k e^{\lambda W_{k+1}},
        \bar\Delta_{\goldelta,\varepsilon}
    \right\}.
\]
Now assume that
\[
    \Delta_k\le\bar\Delta_{\goldelta,\varepsilon}
    \le
    \Delta_{\goldelta,\varepsilon}.
\]
If \(W_{k+1}=+1\), then \(\mathcal I_k\) occurs. By
Lemma~\ref{lem:good-event-sufficient-decrease}, iteration \(k\) is successful.
Hence
\[
    \Delta_{k+1}
    \ge
    \gamma^{-1}\Delta_k
    =
    \Delta_k e^\lambda
    =
    \Delta_k e^{\lambda W_{k+1}}.
\]
If \(W_{k+1}=-1\), then, independently of whether the iteration is successful
or unsuccessful, the update rule gives
\[
    \Delta_{k+1}
    \ge
    \gamma\Delta_k
    =
    \Delta_k e^{-\lambda}
    =
    \Delta_k e^{\lambda W_{k+1}}.
\]
Therefore, in both cases,
\[
    \Delta_{k+1}
    \ge
    \min\left\{
        \Delta_k e^{\lambda W_{k+1}},
        \bar\Delta_{\goldelta,\varepsilon}
    \right\}.
\]
Finally, by Lemma~\ref{lem:good-event-probability},
\[
    \mathbb P(W_{k+1}=+1\mid\mathcal A_k)
    =
    \mathbb P(\mathcal I_k\mid\mathcal A_k)
    \ge
    \beta\left[1-(1-p_\rho\beta)^{\bar{m}}\right]
    =
    q_{\bar{m}}.
\]
If \(\bar{m}\) is chosen so that \(q_{\bar{m}}>1/2\), then
Assumption~\ref{ass:5}(ii) holds with \(q=q_{\bar{m}}\).
\end{proof}

\begin{remark}[Choice of $\bar{\Delta}_{r,\varepsilon}$ and mesh index $j_{r,\varepsilon}$]\label{rem: choice of delta'}
Lemma~\ref{lem:good-event-sufficient-decrease} and Lemma~\ref{lem:good-event-probability} provide an analytic threshold
\(\Delta_{\goldelta,\varepsilon}\) such that, for
\(k<T_{\goldelta,\varepsilon}\), the spherical-cap argument yields a uniformly
positive probability of success whenever
\(
    \Delta_k\le\Delta_{\goldelta,\varepsilon}.
\)
For the stopping-time theorem, we need a threshold belonging to
the geometric mesh generated by the radius updates. Since the updates multiply
the radius by powers of \(\gamma\), the radii belong to the mesh
\[
    \Delta_0\gamma^j,
    \qquad j\in\mathbb Z.
\]
We therefore define
\[
    j_{\goldelta,\varepsilon}
    :=
    \min\left\{
        j\in\mathbb Z:
        \Delta_0\gamma^j\le\Delta_{\goldelta,\varepsilon}
    \right\},
\]
and set
\[
    \bar\Delta_{\goldelta,\varepsilon}
    :=
    \Delta_0\gamma^{j_{\goldelta,\varepsilon}}.
\]
Then
\[
    \bar\Delta_{\goldelta,\varepsilon}
    \le
    \Delta_{\goldelta,\varepsilon},
\]
so Lemma~\ref{lem:good-event-sufficient-decrease} and  Lemma~\ref{lem:good-event-probability} apply whenever
 \(\Delta_k\le\bar\Delta_{\goldelta,\varepsilon}\).
\end{remark}
We are now ready to combine the merit-function drift, the radius recursion, and
the lower bound on the good-event probability within the stopping-time framework.
This gives an expected bound on the number of outer iterations needed to reach
\((\goldelta,\varepsilon)\)-Goldstein stationarity. It is important to highlight that assuming $c_{max}$ is uniformly bounded (see end of Remark \ref{rem:M-lower-bound}) keeps the constants in the
outer-iteration bound of Corollary~\ref{cor:expected-stopping} uniform in \(\varepsilon\).

\begin{corollary}[Expected stopping time]
\label{cor:expected-stopping}
Assume the merit-function drift condition given in  Lemma~\ref{lem:phi_drift_main}
holds with $\Theta:=\eta(1-\gamma^{\Dexp})-c_{\max}>0,$
and let the radius recursion and success-probability bound be as in
Lemma~\ref{lem:good-event-sufficient-decrease} and Lemma~\ref{lem:good-event-probability}. Suppose that \(\bar{m}\) is chosen so that
\[
    q_{\bar{m}}
    :=
    \beta\left[1-(1-p_\rho\beta)^{\bar{m}}\right]
    > q_0>
    \frac12 .
\]
Then we obtain the expected iteration complexity bound
\[
    \mathbb E[T_{\goldelta,\varepsilon}]
    =
    \mathcal O\left(
        \max\left\{
            \goldelta^{-\Dexp},
            \varepsilon^{-\Dexp/(\Dexp-1)}
        \right\}
    \right).
\]

\end{corollary}
\begin{proof}
By Assumption (i) (bounded step sizes) and Lemma~\ref{lem:phi_drift_main} (merit-function drift with $h(\Delta)=\Delta^{\Dexp}$), items (i) and (iii) of Assumptions \ref{ass:5} hold. Lemma~\ref{lem:assumption-ii-holds} verifies item (ii) with the Bernoulli driver $W_{k+1}\in\{-1,+1\}$, parameter $q=q_{\bar{m}}>\tfrac12$, and $\lambda=\log(\gamma^{-1})$, together with the cap at $\bar{\Delta}_{\goldelta, \varepsilon}$. Substituting these constants into the renewal-reward bound (Theorem~\ref{t:expected complexity}) gives the stated inequality. More specifically, Theorem~\ref{t:expected complexity} applies with
\[
    h(\Delta)=\Delta^{\Dexp},
    \qquad
    q=q_{\bar{m}},
    \qquad
    \Delta_{\goldelta,\varepsilon}
    =
    \bar\Delta_{\goldelta,\varepsilon},
\]
and gives
\begin{equation}\label{eq:Cartis}
    \mathbb E[T_{\goldelta,\varepsilon}]
    \le
    \frac{q_{\bar{m}}}{2q_{\bar{m}}-1}
    \frac{\Phi_0}{\Theta\bar\Delta_{\goldelta,\varepsilon}^{\Dexp}}
    +1 .
\end{equation}
By construction of the analytic threshold in Lemma~\ref{lem:good-event-sufficient-decrease}, we have that $\Delta_{\goldelta,\varepsilon}
    \asymp
    \min\left\{
        \goldelta,
        \varepsilon^{1/(\Dexp-1)}
    \right\}.$ 
Therefore, it is easy to see that also
$ \bar\Delta_{\goldelta,\varepsilon}
    \asymp
    \min\left\{
        \goldelta,
        \varepsilon^{1/(\Dexp-1)}
    \right\}.$
It thus follows that
\[
    \frac{1}{\bar\Delta_{\goldelta,\varepsilon}^{\Dexp}}
    =
    \mathcal O\left(
        \max\left\{
            \goldelta^{-\Dexp},
            \varepsilon^{-\Dexp/(\Dexp-1)}
        \right\}
    \right).
\]
Substituting this estimate into~\eqref{eq:Cartis} gives
\[
    \mathbb E[T_{\goldelta,\varepsilon}]
    =
    \mathcal O\left(
        \max\left\{
            \goldelta^{-\Dexp},
            \varepsilon^{-\Dexp/(\Dexp-1)}
        \right\}
    \right),
\]
which proves the claim.

\end{proof}
The previous corollary counts outer iterations. We now translate this bound into
a tested-point complexity estimate by using the fact that, at each iteration, the
algorithm evaluates at most \(\bar{m}\) directions and at most \(\Imax+1\)
extrapolation levels per direction. Also in this case, we note that tightening
the sample-average accuracy to make $c_{max}$ uniformly bounded may increase the raw stochastic sample complexity, but not the tested-point complexity considered in Corollary~\ref{cor:finalcompl}.

\begin{corollary}\label{cor:finalcompl}[Expected tested-point complexity to Goldstein stationarity]
\label{cor:tested-point-Goldstein}
Let \(T_{\goldelta,\varepsilon}\) be the stopping time defined in
\eqref{Teps}. Assume the hypotheses of Corollary~\ref{cor:expected-stopping}. 
Fix $q_0\in(1/2,\beta)$, and choose \(\bar{m}\) large enough so that
\(q_{\bar{m}}\ge q_0\). Suppose moreover that \(\Imax\) is fixed and that the
line search tests at most \(\bar{m}\) directions per outer iteration and at most
\(\Imax+1\) extrapolation levels per direction. Assume $\bar{m}$ is chosen
according to Remark~\ref{rem:M-lower-bound}. Then
\[
    \mathbb E\!\left[
        \sum_{k=0}^{T_{\goldelta,\varepsilon}-1} (1+P_k)
    \right]
    =
    \mathcal O\left(
        \varepsilon^{1-n}
        \max\left\{
            \goldelta^{-\Dexp},
            \varepsilon^{-\Dexp/(\Dexp-1)}
        \right\}
    \right),
\]
where \(P_k\) denotes the number of trial points evaluated during iteration
\(k\), excluding the baseline point \(X_k\).

\end{corollary}
\begin{proof}
At every outer iteration, the algorithm tests at most \(\bar{m}\) directions. For
each direction, the line search evaluates at most the levels
$ i=0,1,\ldots,\Imax.$
Therefore the number of trial points evaluated during iteration \(k\) satisfies
\[
    P_k\le \bar{m}(\Imax+1)
    \qquad\text{a.s.}
\]
Summing up to the stopping time gives
\[
    \sum_{k=0}^{T_{\goldelta,\varepsilon}-1} P_k
    \le
    \bar{m}(\Imax+1)T_{\goldelta,\varepsilon}.
\]
Taking expectations yields
\[
    \mathbb E\!\left[
        \sum_{k=0}^{T_{\goldelta,\varepsilon}-1} P_k
    \right]
    \le
    \bar{m}(\Imax+1)\mathbb E[T_{\goldelta,\varepsilon}].
\]
If the baseline estimate is counted once per outer iteration, the same argument
gives
\[
    \mathbb E\!\left[
        \sum_{k=0}^{T_{\goldelta,\varepsilon}-1} (1+P_k)
    \right]
    \le
    \bigl(1+\bar{m}(\Imax+1)\bigr)
    \mathbb E[T_{\goldelta,\varepsilon}].
\]
The rest follows by substituting the bound from
Corollary~\ref{cor:expected-stopping}. Finally, if \(\bar{m}\) is chosen minimally to ensure
\(q_{\bar{m}}>1/2\), the scaling \(\bar{m}=\mathcal O(\varepsilon^{1-n})\) follows from
Remark~\ref{rem:M-lower-bound}. Thus the result is proved.

\end{proof}

\begin{remark}[Tested points versus raw stochastic oracle samples]
Corollary~\ref{cor:tested-point-Goldstein} counts calls to the estimator
\(\tilde f\), or equivalently the number of tested points. It does not count the
raw stochastic oracle samples used to construct each estimate. If \(\tilde f\) is computed by leveraging a batch size \(m_k\) of function evaluations at
iteration \(k\), then the raw sample cost is proportional to
\[
    \sum_{k<T_{\goldelta,\varepsilon}} m_k P_k,
\]
where \(P_k\) is the number of estimator calls at iteration \(k\). Thus Corollary~\ref{cor:tested-point-Goldstein} should be interpreted as a
tested-point complexity bound, not as a raw stochastic-oracle sample complexity
bound.
\end{remark}

\begin{remark}[Comparison with smooth stochastic direct search]
\label{rem:complexity_comparison_smooth}
To put the expected total tested-point complexity of Corollary \ref{cor:finalcompl} into perspective, consider the symmetric target regime where we seek an $(\varepsilon, \varepsilon)$-Goldstein stationary point (i.e., setting $\goldelta = \varepsilon$). The complexity bound algebraically factors as:
\[
   \mathcal O\Bigl( \varepsilon^{1-n}  \max\bigl\{\varepsilon^{-p},\ \varepsilon^{-\frac{p}{p-1}}\bigr\} \Bigr) 
   = \mathcal O\Bigl( \varepsilon^{1-n} \cdot \varepsilon^{-\frac{p}{\min(p-1, 1)}} \Bigr).
\]
This factorization isolates the theoretical cost of our direct nonsmooth approach. The second term,
\(\mathcal O(\varepsilon^{-p/\min(p-1,1)})\), is exactly the expected iteration-complexity
bound established by \cite{MR4359469} for stochastic directional direct search on smooth
\((C^{1,1})\) objective functions. In that setting, passing from iterations to tested points only
multiplies the bound by the cardinality of the positive spanning set, which is uniformly bounded
(for instance, of order \(n\) for standard positive bases). By contrast, in our nonsmooth random-polling
analysis, such a polling-set factor is replaced by
\(\mathcal O(\varepsilon^{1-n})\). This term represents the geometric cost of sampling enough
uniform random directions to hit a shrinking spherical cap aligned with a Goldstein descent
direction, without the deterministic alignment guarantee provided by the cosine measure of
positive spanning sets in smooth settings.
\end{remark}
It is also useful to make a comparison with randomized-smoothing
zeroth-order methods, such as those of Nesterov and Spokoiny~\cite{MR3627456}
and Kornowski and Shamir~\cite{JMLR:kornowskishamir}.
Those methods replace the original objective \(f\) by a smoothed surrogate
\(f_\mu\) and use random finite-difference or directional-derivative estimators
to approximate gradients of the smoothed function. This approach can lead to
sharper worst-case bounds in regimes where the smoothed problem is an
appropriate proxy for the original one. However, the resulting performance
depends on the choice of the smoothing radius \(\mu\) and on the stochastic
oracle structure. A larger value of \(\mu\) gives a smoother and more stable
surrogate but may introduce a larger bias between \(f_\mu\) and \(f\); a
smaller value of \(\mu\) reduces this bias but can make gradient estimators
more sensitive to stochastic noise. Thus, as \(\mu\) is reduced to improve the
approximation of \(f\) by \(f_\mu\), one may need additional sampling or
variance-control mechanisms to stabilize the estimate.\\
By contrast, the DSE method analyzed here does not optimize a smoothed
surrogate and does not attempt to reconstruct a gradient. It evaluates
candidate points and accepts steps through a sufficient-decrease test on
stochastic function estimates directly. The price paid for this direct
nonsmooth treatment is the worse cap-probability factor in the expected
complexity bound. This factor should however be understood as a worst-case
geometric penalty.
Moreover, for locally Lipschitz functions, points of nondifferentiability are
negligible by Rademacher's theorem  (see, e.g.,
\cite[Section~2.5]{Clarke1990}). Hence random trial points typically lie in
smooth pieces of the objective, where the set of successful directions may be
larger than the single certified cap used in the proof. Thus the true chance of
success can be substantially larger than what the conservative lower bound used
in the analysis says. The advantage in our framework is that no smoothing radius
has to be tuned. This helps explain why, despite the more conservative
worst-case complexity, the method is competitive with random-gradient smoothing
techniques in our empirical benchmarks.

%% file: sections/numerical_results.tex
\section{Numerical Results}
\label{sec:numerics}

We compare our direct search with extrapolation (DSE) against four baselines:
the stochastic direct search (SDS) of~\cite{MR4758401}, StoMADS~\cite{MR4238147},
the zeroth-order gradient-smoothing method (GS) of Nesterov and Spokoiny~\cite{MR3627456},
and the Optimal Stochastic Nonsmooth Nonconvex Optimization Algorithm (OSNNOA)
of Kornowski and Shamir~\cite{JMLR:kornowskishamir}.
Unless specified otherwise, we choose the sufficient-decrease exponent as $\Dexp = 2$ for all methods,
and we use the same budget and profiling protocol as in~\cite{MR4758401}.
The implementation of the proposed method and the code used to reproduce the
numerical experiments are available on Github:  \url{https://github.com/APalmier99/DSE_StDFOpt.git}.

\subsection{DSE Implementation and Solvers Compared}
\label{subsec:impl}

The DSE implementation used in the experiments is a practical variant of the
theoretical scheme analyzed in previous sections.
As in the SDS implementation from~\cite[Section 5]{MR4758401},
we adopt a multi-phase strategy that improves robustness and speed in practice
combining line searches along coordinates with line searches along dense directions.
SDS follows the implementation choices reported in~\cite[Section 5]{MR4758401}.
For StoMADS we use the authors' reference settings. For GS we set parameters as
discussed in~\cite{MR3627456}. For OSNNOA, we implemented the randomized-smoothing
scheme described in~\cite{JMLR:kornowskishamir}, using the proposed stochastic gradient
estimator and parameter choices as guidance.
Amount of noise, oracle budget, starting point, initial stepsizes are fixed across all algorithms for fairness following again the guidelines in~\cite[Section 5]{MR4758401}.

\subsection{Data and Performance Profiles}
\label{subsec:protocol}

Following~\cite{more2009benchmarking}, a run on problem $p$ is declared successful at
iteration $k$ if
\begin{equation}\label{eq:profile-tol}
  f(x_k)\ \le\ f_L\;+\;\tau\bigl(f(x_0)-f_L\bigr),
\end{equation}
where $f_L$ is the best objective value achieved by any solver on $p$.
We use the standard \emph{performance} and \emph{data} profiles,
with time counters $t_{p,s}$ defined as the number of function evaluations used by solver $s$
to satisfy~\eqref{eq:profile-tol}. The performance profile of $s$ is
\[
\rho_s(\alpha)\;=\;\frac{1}{|{\mathcal P}|}\,
\Bigl|\,\bigl\{\,p\in{\mathcal P}:\ 
t_{p,s}\ \le\ \alpha\cdot \min_{s'} t_{p,s'}\ \bigr\}\Bigr|,
\]
and the data profile is
\[
d_s(\kappa)\;=\;\frac{1}{|{\mathcal P}|}\,
\Bigl|\,\bigl\{\,p\in{\mathcal P}:\  t_{p,s}\ \le\ \kappa\,(n_p+1)\ \bigr\}\Bigr|,
\]
where $n_p$ is the dimension of problem $p$~\cite[Section~5]{MR4758401}.
All profiles are computed with the \emph{true} objective values while the solvers operate on their (possibly stochastic) estimates, and the evaluation budget is fixed  as in~\cite[Section 5]{MR4758401}.
We enforce a budget of $10^4(n_p+1)$ function evaluations per run.
We report results for two tolerances $\tau\in\{10^{-2},10^{-4}\}$.\\
To evaluate performances we use a standard nonsmooth DFO test set (unconstrained), aggregating all runs across problems and random seeds when building the profiles, as in~\cite[Section~5]{MR4758401}. The benchmark set comprises 96 instances of nonsmooth DFO problems (see \cite[Table 1]{MR4758401}). To stabilize comparisons, we solve each instance \emph{five} times per solver with different seeds, yielding $|P|=96 \times 5 = 480$ runs per solver.  
We simulated noise after averaging $p_k$ independent samples by adding to the objective $N(0, 1/p_k)$ distributed random variables. The remaining parameters were tuned with a basic grid search.\\
In Figure \ref{fig:profiles-grid}, we report data and performance profiles. It is easy to see that, for both tolerances, DSE outperforms the competitors for the whole range of budgets/ratios. More specifically, 
\begin{itemize}
\item \textbf{Data profiles.} The DSE curve (blue) lies strictly above those of SDS (red),
StoMADS (black),  OSNNOA (china blue) and GS (purple), for essentially
all budgets. At $\tau=10^{-2}$ the gap is pronounced, with DSE solving the largest
fraction of problems at small budgets; Even at $\tau = 10^{-4}$, DSE continues to have a distinct advantage.
\item \textbf{Performance profiles.} DSE attains the highest proportion of ``best-within-a-factor''
wins across ratios $\alpha$. The separation from SDS is visible already near $\alpha\approx 1$,
and remains throughout the range; For the majority of problems, StoMADS, GS, and OSNNOA perform worse.
\end{itemize}
Overall, the profiles indicate that DSE is highly
effective in practice and the multi-phase strategy helps to improve performances: the initial coordinate searches efficiently capture simple descent directions, while the dense directions give strong thorough exploration once stepsizes are small.
\begin{figure}[t]
  \centering

  \begin{subfigure}{0.4\textwidth}
    \centering
    \includegraphics[width=\linewidth]{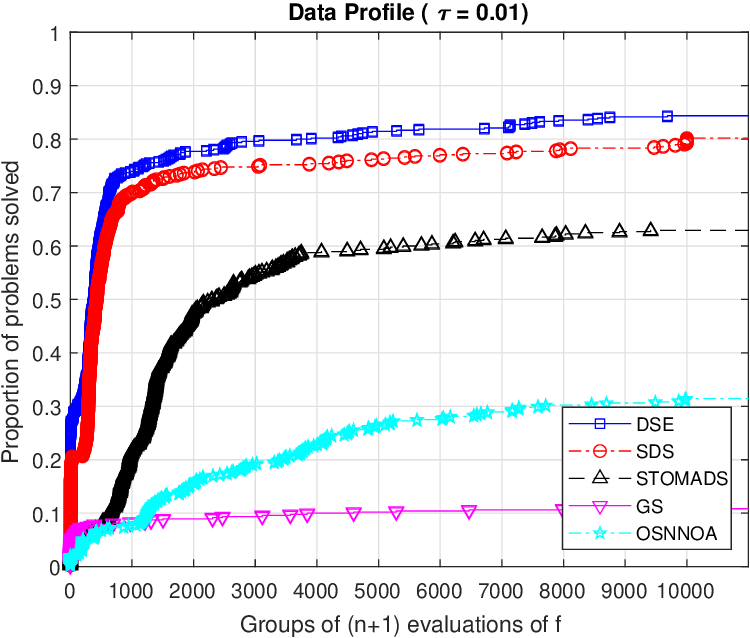}
    \caption{Data profile, $\tau=10^{-2}$}
    \label{fig:data-2}
  \end{subfigure}
  \hspace{3em}
  \begin{subfigure}{0.4\textwidth}
    \centering
    \includegraphics[width=\linewidth]{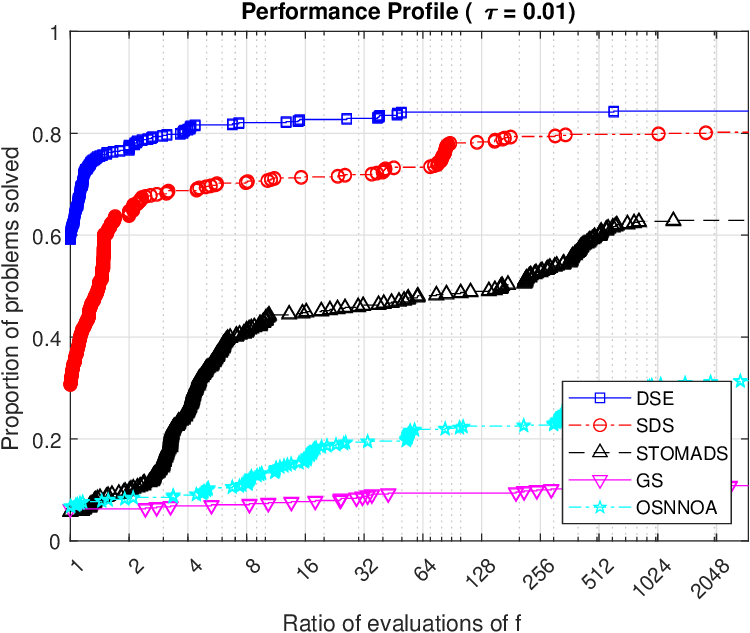}
    \caption{Performance profile, $\tau=10^{-2}$}
    \label{fig:perf-2}
  \end{subfigure}

  \vspace{0.75em}

  \begin{subfigure}{0.4\textwidth}
    \centering
    \includegraphics[width=\linewidth]{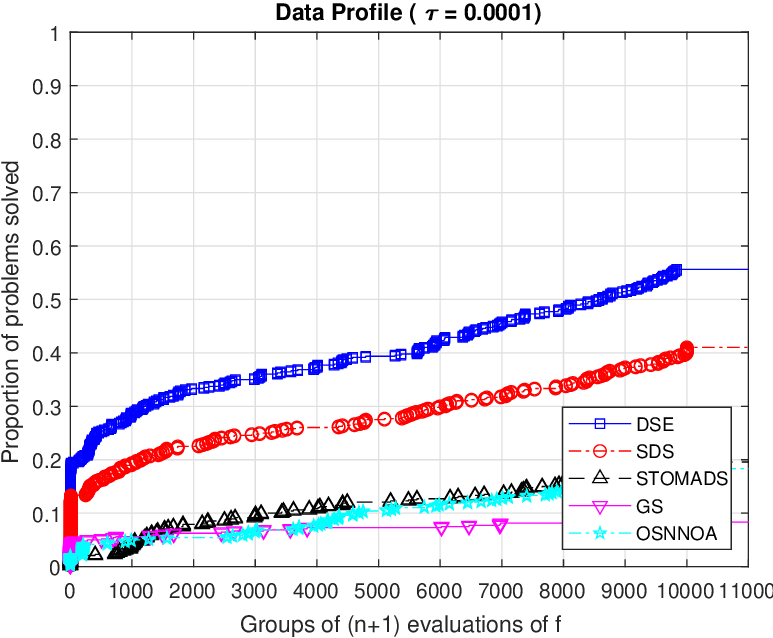}
    \caption{Data profile, $\tau=10^{-4}$}
    \label{fig:data-4}
  \end{subfigure}
  \hspace{3em}
  \begin{subfigure}{0.4\textwidth}
    \centering
    \includegraphics[width=\linewidth]{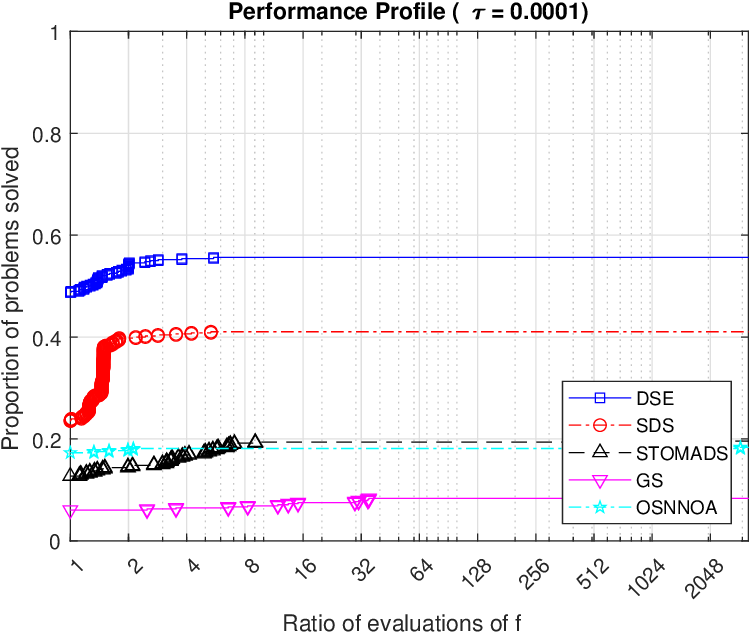}
    \caption{Performance profile, $\tau=10^{-4}$}
    \label{fig:perf-4}
  \end{subfigure}

  \caption{Data (left) and performance (right) profiles. Top row: $\tau=10^{-2}$. Bottom row: $\tau=10^{-4}$.}
  \label{fig:profiles-grid}
\end{figure}

%% file: sections/conclusions.tex
\section{Conclusions and Future Directions}
\label{sec:conclusions}

We proposed and analyzed DSE, a stochastic direct-search method with extrapolation
for unconstrained nonsmooth zeroth-order optimization. The meth\-od combines random
polling directions with a stochastic sufficient-decrease line search and uses only noisy
function-value estimates. Under conditional tail and independence assumptions on the
oracle errors, we proved almost-sure convergence of refining subsequences to Clarke stationary
points.\\
We further established expected complexity bounds for reaching
\((r,\varepsilon)\)-Goldstein stationarity. The proof combines a merit-function
drift argument with a supermartingale stopping-time framework and a
spherical-cap probability estimate for the random polling directions. At the
outer-iteration level, we obtain an expected iteration complexity 
$\mathcal O\left(
        \max\left\{
            r^{-p},
            \varepsilon^{-p/(p-1)}
        \right\}
    \right).$
Moreover, after choosing $\bar{m}$, the maximum number of polling directions to be used at each iteration,  large enough to keep the
success probability uniformly bounded away from \(1/2\), the expected
tested-point complexity satisfies
$\mathcal O\left(
        \varepsilon^{1-n}
        \max\left\{
            r^{-p},
            \varepsilon^{-p/(p-1)}
        \right\}
    \right).$ The numerical results show that extrapolation can substantially improve the practical
performance of stochastic direct-search schemes on nonsmooth benchmark problems.
In particular, the proposed implementation compares favorably with SDS, StoMADS,
and randomized-smoothing baselines on the tested instances.\\
Several directions remain open. First, it would be interesting to sharpen the
dimension dependence in the nonsmooth complexity bound or to identify structural
conditions under which the conservative spherical-cap factor can be improved.  Second, extending the analysis to variance-reduction mechanisms,
common-random-number estimators, and constrained nonsmooth stochastic problems
would further broaden the applicability of the approach.

%% file: sections/appendix.tex
\section{Proofs of the Theoretical Results} \label{appendix:A}

\begin{proof}[Lemma \ref{lem: conditional mean bound}]
Work conditionally on $\mathcal F_{k-1}$ throughout. Introduce the normalized error
\[
Z_{k,i}\ :=\ \frac{|\bar E_{k,i}|}{\Delta_k^{\Dexp}}\ \ge 0 .
\]
Assumption~\ref{ass:quadratic tail} is exactly the conditional tail bound
\[
\mathbb{P}\!\left(Z_{k,i}\ \ge\ \alpha\ \middle|\ \mathcal F_{k-1}\right)\ \le\ \frac{\varepsilon_h}{\alpha^{\aexp}}
\qquad\forall\,\alpha\ge\varepsilon_h.
\]
By the layer-cake representation,
\[
\mathbb{E}\!\left[Z_{k,i}\ \middle|\ \mathcal F_{k-1}\right]
=\int_0^\infty \mathbb{P}\!\left(Z_{k,i}\ge t\ \middle|\ \mathcal F_{k-1}\right)\,dt.
\]
Fix any cutoff $c\ge\varepsilon_h$ and split the integral:
\[
\mathbb{E}[Z_{k,i}\mid\mathcal F_{k-1}]
\ \le\ \int_0^{c} 1\,dt
\ +\ \int_{c}^{\infty} \frac{\varepsilon_h}{t^{\aexp}}\,dt
\ =\ c+\varepsilon_h(\Dexp-1)c^{-1/(\Dexp-1)}.
\]
The right-hand side is minimized (over $c\ge\varepsilon_h$) at
\(
c^\star=\max\{\varepsilon_h^{(\Dexp-1)/\Dexp},\ \varepsilon_h\}
\)
so
\[
\min_{c\ge\varepsilon_h}\Big(c+\varepsilon_h(\Dexp-1)c^{-1/(\Dexp-1)}\Big)
=
\begin{cases}
\Dexp \varepsilon_h^{(\Dexp-1)/\Dexp}, & \varepsilon_h\le 1, \quad \textit{case A}\\[3pt]
\varepsilon_h \left[ 1+ (\Dexp-1)\varepsilon_h^{-1/(\Dexp-1)}\right], & \varepsilon_h>1 \quad \textit{ case B}.
\end{cases}
\]
Therefore
\(
\mathbb{E}[Z_{k,i}\mid\mathcal F_{k-1}]\le \mu_h
\)
with
$\mu_h$ defined as above.
Multiplying back by $\Delta_k^{\Dexp}$ yields
\(
\mathbb{E}[|\bar E_{k,i}|\mid\mathcal F_{k-1}]\le \mu_h\,\Delta_k^{\Dexp}
\),
uniformly in $k$ and $i$, and the result is proved.
\end{proof}

\begin{proof}[Lemma \ref{lem:diff-L1-from-A44}]
Work conditionally on $\mathcal F_{k-1}$ throughout. By the previous lemma, we have
\[
\mathbb E\big[\,|\bar E_{k}^{(-1)}|\,\big|\,\mathcal F_{k-1}\big]\ \le\ \mu_h\,\Delta_k^{\Dexp},
\qquad
\mathbb E\big[\,|\bar E_{k,m}^{(i)}|\,\big|\,\mathcal F_{k-1}\big]\ \le\ \mu_h\,\Delta_k^{\Dexp}
\quad \forall i,\ \forall m.
\]
Using $E_{k,m}^{(i)}=\bar E_{k}^{(-1)}-\bar E_{k,m}^{(i)}$ and the triangle inequality,
\[
\mathbb E\big[\,|E_{k,m}^{(i)}|\,\big|\,\mathcal F_{k-1}\big]
\ \le\ \mathbb E\big[\,|\bar E_{k}^{(-1)}|\,\big|\,\mathcal F_{k-1}\big]
      + \mathbb E\big[\,|\bar E_{k,m}^{(i)}|\,\big|\,\mathcal F_{k-1}\big]
\ \le\ 2\,\mu_h\,\Delta_k^{\Dexp},
\]
which is \eqref{eq:diff-L1-core}. The signed bound \eqref{eq:diff-signed} follows since
$|\mathbb E[E_{k,m}^{(i)}\mid\mathcal F_{k-1}]|\le \mathbb E[|E_{k,m}^{(i)}|\mid\mathcal F_{k-1}]$ by Jensen inequality.
\end{proof}
\noindent
Finally, bounding $E^{\max}_k \leq \sum_{m}\sum_{i}|E_{k,m}^{(i)}|$ and applying the previous result, immediately yields Lemma \ref{ass:postL1} as well.

%% file: sections/appendix_estimates.tex
\section{Random Estimates Construction}
\label{appendix:random estimates}

Standard sampling strategies in stochastic derivative-free optimization often rely on averaging multiple independent realizations of the stochastic oracle. Under the classical assumptions commonly imposed on such sample-average estimators, one can establish the stronger oracle conditions typically assumed in the stochastic derivative-free optimization literature, which in turn imply Assumption~\ref{ass:quadratic tail}. The purpose of this appendix is to make this connection explicit and derive the corresponding batch-size requirements.
\begin{proposition}[Rosenthal batch size for Assumption~\ref{ass:quadratic tail}]
\label{prop:rosenthal-ass33}
Let \(p\in(1,2]\) and set
\[
    r:=\frac{p}{p-1}\ge2.
\]
Let \(\mathcal H\) be a sigma-algebra, let \(X\) be an
\(\mathcal H\)-measurable random point, and let \(\Delta>0\) be an
\(\mathcal H\)-measurable random variable.
Suppose that, conditionally on \(\mathcal H\), the oracle samples
\[
    F(X,\zeta_1),\ldots,F(X,\zeta_W)
\]
are independent and satisfy, for some \(\sigma>0\),
\[
    \mathbb E\!\left[
        F(X,\zeta_j)-f(X)
        \,\middle|\,
        \mathcal H
    \right]
    =0, 
    \qquad 
    \mathbb E\!\left[
        \left|F(X,\zeta_j)-f(X)\right|^r
        \,\middle|\,
        \mathcal H
    \right]
    \le \sigma^r,
\]
for every \(j=1,\ldots,W\), with \(W\) a positive, integer-valued,
\(\mathcal H\)-measurable random variable.\\
Define the sample-average estimator
\[
    \tilde f(X)
    :=
    \frac{1}{W}
    \sum_{j=1}^{W}F(X,\zeta_j).
\]
Then there exists a constant \(C_r>0\), depending only on \(r\), such that,
if
\[
    W
    \ge
    \left(
        \frac{C_r\sigma^r}{\varepsilon_h}
    \right)^{2/r}
    \Delta^{-2p},
\]
then
\[
    \mathbb P\!\left(
        |\tilde f(X)-f(X)|
        \ge
        \alpha\Delta^p
        \,\middle|\,
        \mathcal H
    \right)
    \le
    \frac{\varepsilon_h}{\alpha^{p/(p-1)}}
    \qquad
    \text{for every }\alpha>0.
\]
Consequently, a batch size
 $W=\mathcal O(\Delta^{-2p})$
is sufficient to satisfy Assumption~\ref{ass:quadratic tail}.
\end{proposition}
\begin{proof}
Define
\[
    \xi_j:=F(X,\zeta_j)-f(X),
    \qquad j=1,\ldots,W.
\]
Then
\[
    \tilde f(X)-f(X)
    =
    \frac{1}{W}\sum_{j=1}^{W}\xi_j.
\]
Conditionally on \(\mathcal H\), the random variables
\(\xi_1,\ldots,\xi_W\) are independent and centered, and satisfy
\[
    \mathbb E\!\left[
        |\xi_j|^r
        \,\middle|\,
        \mathcal H
    \right]
    \le\sigma^r.
\]
Since \(r\ge2\), the conditional version of Rosenthal's inequality gives
\[
\begin{aligned}
    \mathbb E\!\left[
        \left|
        \sum_{j=1}^{W}\xi_j
        \right|^r
        \,\middle|\,
        \mathcal H
    \right]
    \le C_r\sigma^r W^{r/2},
\end{aligned}
\]
where \(C_r>0\) depends only on \(r\).
Therefore,
\[
    \mathbb E\!\left[
        |\tilde f(X)-f(X)|^r
        \,\middle|\,
        \mathcal H
    \right]
    \le
    C_r\sigma^r W^{-r/2}.
\]
By conditional Markov's inequality,
\[
\begin{aligned}
    &\mathbb P\!\left(
        |\tilde f(X)-f(X)|
        \ge
        \alpha\Delta^p
        \,\middle|\,
        \mathcal H
    \right)
    \le
    \frac{
        C_r\sigma^r W^{-r/2}
    }{
        \alpha^r\Delta^{pr}
    }.
\end{aligned}
\]
If
\[
    W
    \ge
    \left(
        \frac{C_r\sigma^r}{\varepsilon_h}
    \right)^{2/r}
    \Delta^{-2p},
\]
then
\[
    C_r\sigma^r W^{-r/2}\Delta^{-pr}
    \le\varepsilon_h.
\]
Hence
\[
    \mathbb P\!\left(
        |\tilde f(X)-f(X)|
        \ge
        \alpha\Delta^p
        \,\middle|\,
        \mathcal H
    \right)
    \le
    \frac{\varepsilon_h}{\alpha^r}.
\]
Since \(r=p/(p-1)\), the claim follows.
\end{proof}
\begin{remark}[Application to the algorithm]
At iteration \(k\), Proposition~\ref{prop:rosenthal-ass33} is applied with
\[
    \mathcal H=\mathcal F_{k-1},
    \qquad
    \Delta=\Delta_k.
\]
For the base-point estimate, take \(X=X_k\), while for the trial-point estimate associated with direction \(m\) and extrapolation level \(i\), take $X=X_{k,m}^{(i)}.$
Thus, define
\[
    \tilde f_k^{(-1)}
    :=
    \frac{1}{W_k^{(-1)}}
    \sum_{j=1}^{W_k^{(-1)}}
    F\bigl(X_k,\zeta_{k,j}^{(-1)}\bigr)
\]
and
\[
    \tilde f_{k,m}^{(i)}
    :=
    \frac{1}{W_{k,m}^{(i)}}
    \sum_{j=1}^{W_{k,m}^{(i)}}
    F\bigl(X_{k,m}^{(i)},\zeta_{k,m,j}^{(i)}\bigr).
\]
Here,
\[
    \left\{\zeta_{k,j}^{(-1)}\right\}_{j}
    \quad\text{and}\quad
    \left\{\zeta_{k,m,j}^{(i)}\right\}_{m,i,j}
\]
denote conditionally independent copies of the generic oracle random variable \(\zeta\), given \(\mathcal F_{k-1}\).\\
These estimates satisfy Assumption~\ref{ass:quadratic tail} whenever
\[
    W_k^{(-1)}
    =
    \mathcal O(\Delta_k^{-2p}),
    \qquad
    W_{k,m}^{(i)}
    =
    \mathcal O(\Delta_k^{-2p}),
\]
uniformly over \(m\in[1:\bar m]\) and \(i\in[0:\Imax]\).\\
Moreover, using fresh conditionally independent batches for the base point and for all potential trial points ensures that
Assumption~\ref{ass:indep} is also satisfied.
\end{remark}
\begin{remark}
For $p=2$, one has $r=2$, and the result reduces to the usual finite-variance
case with
$W_k=\mathcal O(\Delta_k^{-4}).$
For $p\in(1,2)$, one has $r>2$, so a moment assumption stronger than finite
variance is needed in order to obtain the tail exponent required by
Assumption~\ref{ass:quadratic tail}.
\end{remark}

%% file: sdse.bib
@article {MR4959972,
    AUTHOR = {Ha, Yunsoo and Shashaani, Sara and Pasupathy, Raghu},
     TITLE = {Complexity of zeroth- and first-order stochastic trust-region
              algorithms},
   JOURNAL = {SIAM J. Optim.},
  FJOURNAL = {SIAM Journal on Optimization},
    VOLUME = {35},
      YEAR = {2025},
    NUMBER = {3},
     PAGES = {2098--2127},
      ISSN = {1052-6234,1095-7189},
   MRCLASS = {90C15 (90C55)},
  MRNUMBER = {4959972},
       DOI = {10.1137/24M1664484},
       URL = {https://doi.org/10.1137/24M1664484},
}

@article {MR3880261,
    AUTHOR = {Shashaani, Sara and Hashemi, Fatemeh S. and Pasupathy, Raghu},
     TITLE = {A{STRO}-{DF}: a class of adaptive sampling trust-region
              algorithms for derivative-free stochastic optimization},
   JOURNAL = {SIAM J. Optim.},
  FJOURNAL = {SIAM Journal on Optimization},
    VOLUME = {28},
      YEAR = {2018},
    NUMBER = {4},
     PAGES = {3145--3176},
      ISSN = {1052-6234,1095-7189},
   MRCLASS = {90C56 (90C15)},
  MRNUMBER = {3880261},
MRREVIEWER = {Wim\ van Ackooij},
       DOI = {10.1137/15M1042425},
       URL = {https://doi.org/10.1137/15M1042425},
}

@article {MR3627456,
    AUTHOR = {Nesterov, Yurii and Spokoiny, Vladimir},
     TITLE = {Random gradient-free minimization of convex functions},
   JOURNAL = {Found. Comput. Math.},
  FJOURNAL = {Foundations of Computational Mathematics. The Journal of the
              Society for the Foundations of Computational Mathematics},
    VOLUME = {17},
      YEAR = {2017},
    NUMBER = {2},
     PAGES = {527--566},
      ISSN = {1615-3375,1615-3383},
   MRCLASS = {90C25 (68Q25 90C15 90C56)},
  MRNUMBER = {3627456},
       DOI = {10.1007/s10208-015-9296-2},
       URL = {https://doi.org/10.1007/s10208-015-9296-2},
}

@article {MR4919092,
    AUTHOR = {Dzahini, K. J. and Rinaldi, F. and Royer, C. W. and Zeffiro,
              D.},
     TITLE = {Direct-search methods in the year 2025: theoretical guarantees
              and algorithmic paradigms},
   JOURNAL = {EURO J. Comput. Optim.},
  FJOURNAL = {EURO Journal on Computational Optimization},
    VOLUME = {13},
      YEAR = {2025},
     PAGES = {Paper No. 100110, 24},
      ISSN = {2192-4406,2192-4414},
   MRCLASS = {90-02 (65K05 90C30)},
  MRNUMBER = {4919092},
       DOI = {10.1016/j.ejco.2025.100110},
       URL = {https://doi.org/10.1016/j.ejco.2025.100110},
}

@article {MR4758401,
    AUTHOR = {Rinaldi, F. and Vicente, L. N. and Zeffiro, D.},
     TITLE = {Stochastic trust-region and direct-search methods: a weak tail
              bound condition and reduced sample sizing},
   JOURNAL = {SIAM J. Optim.},
  FJOURNAL = {SIAM Journal on Optimization},
    VOLUME = {34},
      YEAR = {2024},
    NUMBER = {2},
     PAGES = {2067--2092},
      ISSN = {1052-6234,1095-7189},
   MRCLASS = {90C26 (90C15 90C56)},
  MRNUMBER = {4758401},
       DOI = {10.1137/22M1543446},
       URL = {https://doi.org/10.1137/22M1543446},
}

@article {MR4359469,
    AUTHOR = {Dzahini, Kwassi Joseph},
     TITLE = {Expected complexity analysis of stochastic direct-search},
   JOURNAL = {Comput. Optim. Appl.},
  FJOURNAL = {Computational Optimization and Applications. An International
              Journal},
    VOLUME = {81},
      YEAR = {2022},
    NUMBER = {1},
     PAGES = {179--200},
      ISSN = {0926-6003,1573-2894},
   MRCLASS = {90C15 (90C26 90C56)},
  MRNUMBER = {4359469},
MRREVIEWER = {Dmitri\ E.\ Kvasov},
       DOI = {10.1007/s10589-021-00329-9},
       URL = {https://doi.org/10.1007/s10589-021-00329-9},
}

@article {MR4777851,
    AUTHOR = {Dzahini, K. J. and Wild, S. M.},
     TITLE = {Stochastic trust-region algorithm in random subspaces with
              convergence and expected complexity analyses},
   JOURNAL = {SIAM J. Optim.},
  FJOURNAL = {SIAM Journal on Optimization},
    VOLUME = {34},
      YEAR = {2024},
    NUMBER = {3},
     PAGES = {2671--2699},
      ISSN = {1052-6234,1095-7189},
   MRCLASS = {90C15 (90C55 90C56)},
  MRNUMBER = {4777851},
MRREVIEWER = {Morteza\ Kimiaei},
       DOI = {10.1137/22M1524072},
       URL = {https://doi.org/10.1137/22M1524072},
}

@article {MR4060460,
    AUTHOR = {Paquette, Courtney and Scheinberg, Katya},
     TITLE = {A stochastic line search method with expected complexity
              analysis},
   JOURNAL = {SIAM J. Optim.},
  FJOURNAL = {SIAM Journal on Optimization},
    VOLUME = {30},
      YEAR = {2020},
    NUMBER = {1},
     PAGES = {349--376},
      ISSN = {1052-6234,1095-7189},
   MRCLASS = {90C30 (90C15)},
  MRNUMBER = {4060460},
       DOI = {10.1137/18M1216250},
       URL = {https://doi.org/10.1137/18M1216250},
}

@article {MR4151319,
    AUTHOR = {Blanchet, Jose and Cartis, Coralia and Menickelly, Matt and
              Scheinberg, Katya},
     TITLE = {Convergence rate analysis of a stochastic trust-region method
              via supermartingales},
   JOURNAL = {INFORMS J. Optim.},
  FJOURNAL = {INFORMS Journal on Optimization},
    VOLUME = {1},
      YEAR = {2019},
    NUMBER = {2},
     PAGES = {92--119},
      ISSN = {2575-1484,2575-1492},
   MRCLASS = {90C15 (60G40 60G42 90C55)},
  MRNUMBER = {4151319},
MRREVIEWER = {I.\ M.\ Stancu-Minasian},
       DOI = {10.1287/ijoo.2019.0016},
       URL = {https://doi.org/10.1287/ijoo.2019.0016},
}

@article {MR3800867,
    AUTHOR = {Chen, R. and Menickelly, M. and Scheinberg, K.},
     TITLE = {Stochastic optimization using a trust-region method and random
              models},
   JOURNAL = {Math. Program.},
  FJOURNAL = {Mathematical Programming},
    VOLUME = {169},
      YEAR = {2018},
    NUMBER = {2},
     PAGES = {447--487},
      ISSN = {0025-5610,1436-4646},
   MRCLASS = {90C15 (90C30 90C56)},
  MRNUMBER = {3800867},
MRREVIEWER = {I.\ M.\ Stancu-Minasian},
       DOI = {10.1007/s10107-017-1141-8},
       URL = {https://doi.org/10.1007/s10107-017-1141-8},
}

@article {MR3245880,
    AUTHOR = {Bandeira, A. S. and Scheinberg, K. and Vicente, L. N.},
     TITLE = {Convergence of trust-region methods based on probabilistic
              models},
   JOURNAL = {SIAM J. Optim.},
  FJOURNAL = {SIAM Journal on Optimization},
    VOLUME = {24},
      YEAR = {2014},
    NUMBER = {3},
     PAGES = {1238--1264},
      ISSN = {1052-6234,1095-7189},
   MRCLASS = {90C56 (90C30)},
  MRNUMBER = {3245880},
MRREVIEWER = {Saman\ Babaie-Kafaki},
       DOI = {10.1137/130915984},
       URL = {https://doi.org/10.1137/130915984},
}

@article {MR4238147,
    AUTHOR = {Audet, Charles and Dzahini, Kwassi Joseph and Kokkolaras,
              Michael and Le Digabel, S\'ebastien},
     TITLE = {Stochastic mesh adaptive direct search for blackbox
              optimization using probabilistic estimates},
   JOURNAL = {Comput. Optim. Appl.},
  FJOURNAL = {Computational Optimization and Applications. An International
              Journal},
    VOLUME = {79},
      YEAR = {2021},
    NUMBER = {1},
     PAGES = {1--34},
      ISSN = {0926-6003,1573-2894},
   MRCLASS = {90C56 (60G42 90C15)},
  MRNUMBER = {4238147},
MRREVIEWER = {Suvra\ Kanti\ Chakraborty},
       DOI = {10.1007/s10589-020-00249-0},
       URL = {https://doi.org/10.1007/s10589-020-00249-0},
}

@book {MR2487816,
    AUTHOR = {Conn, Andrew R. and Scheinberg, Katya and Vicente, Luis N.},
     TITLE = {Introduction to derivative-free optimization},
    SERIES = {MPS/SIAM Series on Optimization},
    VOLUME = {8},
 PUBLISHER = {Society for Industrial and Applied Mathematics (SIAM),
              Philadelphia, PA; Mathematical Programming Society (MPS),
              Philadelphia, PA},
      YEAR = {2009},
     PAGES = {xii+277},
      ISBN = {978-0-898716-68-9},
   MRCLASS = {90-02 (65K05 90C30 90C56)},
  MRNUMBER = {2487816},
MRREVIEWER = {Olga\ Brezhneva},
       DOI = {10.1137/1.9780898718768},
       URL = {https://doi.org/10.1137/1.9780898718768},
}

@article {MR3963507,
    AUTHOR = {Larson, Jeffrey and Menickelly, Matt and Wild, Stefan M.},
     TITLE = {Derivative-free optimization methods},
   JOURNAL = {Acta Numer.},
  FJOURNAL = {Acta Numerica},
    VOLUME = {28},
      YEAR = {2019},
     PAGES = {287--404},
      ISSN = {0962-4929,1474-0508},
   MRCLASS = {90C56 (65Kxx)},
  MRNUMBER = {3963507},
MRREVIEWER = {Dmitri\ E.\ Kvasov},
       DOI = {10.1017/s0962492919000060},
       URL = {https://doi.org/10.1017/s0962492919000060},
}

@article {MR3506227,
    AUTHOR = {Larson, Jeffrey and Billups, Stephen C.},
     TITLE = {Stochastic derivative-free optimization using a trust region
              framework},
   JOURNAL = {Comput. Optim. Appl.},
  FJOURNAL = {Computational Optimization and Applications. An International
              Journal},
    VOLUME = {64},
      YEAR = {2016},
    NUMBER = {3},
     PAGES = {619--645},
      ISSN = {0926-6003,1573-2894},
   MRCLASS = {90C56 (62L20)},
  MRNUMBER = {3506227},
MRREVIEWER = {Ulf\ Lorenz},
       DOI = {10.1007/s10589-016-9827-z},
       URL = {https://doi.org/10.1007/s10589-016-9827-z},
}

@article {MR4550962,
    AUTHOR = {Dzahini, Kwassi Joseph and Kokkolaras, Michael and Le Digabel,
              S\'ebastien},
     TITLE = {Constrained stochastic blackbox optimization using a
              progressive barrier and probabilistic estimates},
   JOURNAL = {Math. Program.},
  FJOURNAL = {Mathematical Programming},
    VOLUME = {198},
      YEAR = {2023},
    NUMBER = {1},
     PAGES = {675--732},
      ISSN = {0025-5610,1436-4646},
   MRCLASS = {90C15 (90C30 90C56)},
  MRNUMBER = {4550962},
MRREVIEWER = {Kurt\ Marti},
       DOI = {10.1007/s10107-022-01787-7},
       URL = {https://doi.org/10.1007/s10107-022-01787-7},
}

@misc{desantis2025linesearchbasedderivativefreemethodnoisy,
      title={A linesearch-based derivative-free method for noisy black-box problems}, 
      author={Alberto De Santis and Giampaolo Liuzzi and Stefano Lucidi},
      year={2025},
      eprint={2508.00495},
      archivePrefix={arXiv},
      primaryClass={math.OC},
      url={https://arxiv.org/abs/2508.00495}, 
}

@book{Clarke1990,
  title = {Optimization and Nonsmooth Analysis},
  publisher = {Society for Industrial and Applied Mathematics},
  author = {Clarke,  Frank H.},
  year = {1990} 
}

@book{Jahn1996,
  title = {Introduction to the Theory of Nonlinear Optimization},
  publisher = {Springer Berlin Heidelberg},
  author = {Jahn,  Johannes},
  year = {1996}
}

@book{MardiaJupp2000,
  author    = {Mardia, Kanti V. and Jupp, Peter E.},
  title     = {Directional Statistics},
  publisher = {John Wiley \& Sons},
  address   = {Chichester},
  year      = {2000}
}

@article{more2009benchmarking,
  title={Benchmarking derivative-free optimization algorithms},
  author={Mor{\'e}, Jorge J and Wild, Stefan M},
  journal={SIAM Journal on Optimization},
  volume={20},
  number={1},
  pages={172--191},
  year={2009},
  publisher={SIAM}
}

@article{JMLR:kornowskishamir,
  author  = {Guy Kornowski and Ohad Shamir},
  title   = {An Algorithm with Optimal Dimension-Dependence for Zero-Order Nonsmooth Nonconvex Stochastic Optimization},
  journal = {Journal of Machine Learning Research},
  year    = {2024},
  volume  = {25},
  number  = {122},
  pages   = {1--14},
  url     = {http://jmlr.org/papers/v25/23-1159.html}
}

@article{Goldstein1977,
  author  = {Goldstein, A. A.},
  title   = {Optimization of Lipschitz Continuous Functions},
  journal = {Mathematical Programming},
  volume  = {13},
  number  = {1},
  pages   = {14--22},
  year    = {1977}
}

@inproceedings{LinEtAl2022GFM,
  author    = {Lin, Tianyi and Zheng, Zeyu and Jordan, Michael I.},
  title     = {Gradient-Free Methods for Deterministic and Stochastic Nonsmooth Nonconvex Optimization},
  booktitle = {Advances in Neural Information Processing Systems},
  year      = {2022}
}

@article{GrattonRoyerVicenteZhang2015,
  author  = {Gratton, Serge and Royer, Clément W. and Vicente, Luís N. and Zhang, Zaikun},
  title   = {Direct Search Based on Probabilistic Descent},
  journal = {SIAM Journal on Optimization},
  volume  = {25},
  number  = {3},
  pages   = {1515--1541},
  year    = {2015},
  doi     = {10.1137/140961602}
}
